\documentclass[11pt]{amsart}
\usepackage{amsmath,amssymb,amsthm,hyperref,stmaryrd}
\usepackage{tikz}  
 \newtheorem{theorem}{Theorem}[section]

 \newtheorem{lemma}[theorem]{Lemma}
 \newtheorem{proposition}[theorem]{Proposition}
 
 {\theoremstyle{definition}
 
 \newtheorem{remark}[theorem]{Remark}

 }

\begin{document}
\title[Discretisation error and local times]{Discretisation error for stochastic integrals with respect to the fractional Brownian motion with discontinuous integrands and local times}

\author[V. Garino]{Valentin Garino}
\address{Uppsala University, Department of Mathematics, 751 06, Uppsala, Sweden}
\email{valentin.garino@math.uu.se}
\author[L. Viitasaari]{Lauri Viitasaari}
\address{Uppsala University, Department of Mathematics, 751 06, Uppsala, Sweden}
\email{lauri.viitasaari@math.uu.se}

\date{\today}

\begin{abstract}
We consider equidistant Riemann approximations of stochastic integrals $\int_0^T f(B^H_s)dB^H_s$ with respect to the fractional Brownian motion with $H>\frac12$, where $f$ is an arbitrary function of locally bounded variation, hence possibly possessing discontinuities. We prove that properly normalised approximation error converge in the $L^2$-topology to a functional of the local time, and we provide rate of convergence for this approximation. As such, our results complements some recent advances on the topic as well as provides new methods for simulation of local times.
\end{abstract}

\maketitle

\medskip\noindent
{\bf Mathematics Subject Classifications (2020)}: 
60G15, 60G22, 60H05, 26A33

\medskip\noindent
{\bf Keywords:} 
approximation of stochastic integral,
discontinuous integrands,
sharp rate of convergence,
local time,
fractional Brownian motions

\allowdisplaybreaks

\section{Introduction}
Rate of convergence for approximation error arising in discretisations of stochastic integrals have numerous applications and is hence an interesting subject of research. Indeed, such approximation error bounds allow to quantify how accurate simulations are or derive finite sample bounds for parameter estimates in various stochastic models, when only discrete observations are available.

The topic has been a topic of active research in is already relatively well-understood in the context of It\^o integrals, see e.g. \cite{Kloeden-Platen}. Later on the focus shifted to the case of fractional Brownian motion and related processes. In this case, the integrals can be understood in pathwise sense, by means of Young integration \cite{young1936inequality} for $H>\frac12$ or by means of rough paths \cite{Friz-Hairer} for $\frac14<H<\frac12$. The rate of convergence in the Riemann approximation of $\int_0^T f(B^H_u)dB^H_u$ in the case of fractional Brownian motion with $H>\frac12$ and (sufficiently) smooth $f$ is known to be proportional $n^{1-2H}$, see \cite{garino2022asymptotic} and refences therein. However, when $f$ possess discontinuities, the problem is more subtle as the concept of integration requires some additional work, see e.g. \cite{chen2019pathwise, HTV1,HTV2}, and it is not clear whether one obtains the same rate $n^{1-2H}$ in the case of discontinuous integrand. Indeed, even in the Brownian motion case, introducing discontinuities to $f$ reduces the rate from $n^{-1/2}$ to $n^{-1/4}$, see \cite{azmoodeh2015rate} where it was also proved that in the fBm case one has an upper bound proportional to $n^{1/2-H}$ which would again be half of the rate $n^{1-2H}$ of the smooth case. On the other hand, later on in \cite{chen2019pathwise}, surprisingly the rate was proved to be proportional to $n^{1-2H}$, that is in contrast to the expectations arising from the Brownian case. Finally, we mention a recent preprint \cite{azmoodeh2022sharp}, where the authors studied the sharp rate of convergence for the expectation in the presence of discontinuities, verifying that indeed the rate $n^{1-2H}$ is the correct one also in the presence of jumps. 

In this article we continue the work \cite{azmoodeh2022sharp} by proving stronger mode of convergence, together with the second order rate of convergence. That is, we prove that properly $n^{1-2H}$-normalised error convergences towards a (functional of) local time of the fractional Brownian motion, with a rate proportional to $n^{(H-1)/2}$. As a by-product, our result allows for simulations of the local time from a sample path of the process $B^H$. Obviously, we recover known results for more regular integrand $f$ as a particular case. Our proof is based on a technical estimate how expectations of functions of increments of $B^H$ can be approximated by using division between large and small increments, and this result may have independent interest with applications to other problems as well, see Proposition \ref{mainBound}.  

In order to present our main result, recall that when a real valued function is convex, its second derivative in the distributional sense can be identified to a positive and $\sigma$-finite Radon measure. 
The following result is the main finding of the present paper.

\begin{theorem}\label{main}
Let $B=(B_1,B_2)$ be a two dimensional fractional Brownian motion (with $B^1$ and $B^2$ independant) with Hurst index $H\in \left(\frac{1}{2},1\right)$. 
Let $f$ be the left derivative of a real valued convex function. 
For $(i,j)\in\llbracket 1,2\rrbracket^2$, $n\in\mathbb{N}^*$ and $t\in[0,T]$, let

$$S^{i,j}_n(t):=n^{2H-1}\left(\int_0^tf(B^i_s)dB^j_s-\sum_{k=0}^{\lfloor nt\rfloor}f\left(B^i_{\frac{k}{n}}\right)\left(B^j_{\frac{k+1}{n}\wedge t}-B^j_\frac{k}{n}\right)\right).$$

Then, there is a constant $P>0$ depending only on $T,H$ such that if the distributional derivative $\mu=f'$ verifies the growth condition

\begin{equation}\label{growthC}\int_{\mathbb{R}}e^{-P\frac{a^2}{2}}d\mu(a)<\infty,\end{equation}

then 

\begin{equation}\label{mainEq}
\left\|S_n^{i,j}(t)-\delta_{ij}\int_{\mathbb{R}}L^i_t(a)d\mu(a)\right\|_{L^2(\Omega)}\leq C\int_{\mathbb{R}}e^{-P\frac{a^2}{2}}d\mu(a)n^{-\frac{1-H}{2}},
\end{equation}

where $L^i$ is the local time of $B^i$, $\delta_{ij}$ is the Kronecker symbol and $C=C(H,T)$ is a constant depending only on $H,T$.
\end{theorem}
\begin{remark}
\label{remark:BV}
It is immediate from our proof (see Remark \ref{remark:BV-proof}) that the result remains valid for any function $f\in BV_{loc}$, where $BV_{loc}$ denotes the space of functions that are of locally bounded variation. In this case the distributional derivative $\mu = f'$ is a signed measure with total variation measure $d|\mu|$, and the growth condition \eqref{growthC} is replaced by 
$$
\int_{\mathbb{R}}e^{-P\frac{a^2}{2}}d|\mu|(a)<\infty,
$$
and the integral on the right hand side of \eqref{mainEq} is with respect to $d|\mu|$.
\end{remark}

A particular case of the Theorem \ref{main} leads to the following result:

\begin{proposition}\label{particular}
Let $B$ be a one dimensional fractional Brownian motion of Hurst index $H>\frac{1}{2}$.
For all $a\in\mathbb{R}$, we have 

\begin{eqnarray}\label{particularEq}
&&\left\|2n^{2H-1}\sum_{k=0}^{\lfloor nt\rfloor}\left|B_{\frac{k+1}{n}\wedge t}\right|\mathbb{I}_{\left\{sgn\left(B_{\frac{k+1}{n}\wedge t}-a\right)sgn\left(B_{\frac{k}{n}}-a\right)<0\right\}}-L_t(a)\right\|_{L^2(\Omega)}\notag\\
&\leq& C(H,T)e^{-P\frac{a^2}{2}}n^{-\frac{1-H}{2}}\notag\\
\end{eqnarray}

\end{proposition}

The present paper completes a study initiated in \cite{azmoodeh2015rate} and pursued in \cite{azmoodeh2022sharp} and \cite{chen2019pathwise}.
Proposition \ref{particular} was initially proved in the case of the fractional Brownian motion of index $H<\frac{1}{2}$ (see \cite{mukeru2017representation}).
 Closely related results where proved in the recent papers \cite{matsuda2022extension} and \cite{podolskij2018comment}. Notice that none of the aforementioned
papers provide a quantitative bound of the form \eqref{particularEq}. On the other hand, the rate $\frac{1-H}{2}$ in the case $H<\frac{1}{2}$ already appears in
the recent paper \cite{altmeyer2021approximation}, which looks at various approximations of the local time (see \cite{altmeyer2021approximation}, Corollary 7).

\begin{remark} When one assume that $f$ is for example a $\mathcal{C}^1$ function, then we can write $d\mu=f'd\lambda$ with $\lambda$ the Lebesgue measure. 
Then, we can use the well known identity
$$\int_{\mathbb{R}}L^i_t(a)d\mu(a)=\int_{\mathbb{R}}L^i_t(a)f'(a)da=\int_0^tf'(B^i_s)ds,$$
and then recover existing results in the smooth case (see e.g \cite{garino2022asymptotic}).
\end{remark}

\begin{remark}\label{sharpness} When the function $f$ is not continous (take for example $x\rightarrow sgn(x)$), the rate $\frac{1-H}{2}$ (which we expect to be sharp)
in the r.h.s of \eqref{mainEq} is in variance with the regular case $f\in\mathcal{C}^1$, where we have (see \cite{garino2022asymptotic}, Corollary 3.4):

$$S_n^{i,j}(t)-\frac{\delta_{ij}}{2}\int_0^tf'(B^i)ds=O_{\mathbb{P}}\left(\kappa_H(n)\right),$$

where $\kappa_H$ is given by:

\begin{equation*}\kappa_H(n):=\left\{
      \begin{array}{cll}
        n^{-\frac{1}{2}} &\textit{ if } H\in(\frac12,\frac34)\\
        \left(n\ln(n)\right)^{-\frac{1}{2}}& \textit{ if } H=\frac{3}{4}\\
        n^{2H-2} &\textit{ if } H\in(\frac34,1)\\
      \end{array}
    \right.
\end{equation*}

In any case, we can see that $\kappa_H(n)n^{\frac{1-H}{2}}\underset{n\rightarrow\infty}\longrightarrow 0$. Investigating the asymptotic behaviour of the second order error
$$\epsilon(n):=n^{\frac{1-H}{2}}\left(S_n^{i,j}-\frac{\delta_{ij}}{2}\int_0^tL_t^i(a)d\mu(a)\right)$$
could be a further interesting problem.

\end{remark}

The main technical result required in the proof of Theorem \ref{main} is the quantitative bound stated in Proposition \ref{mainBound}, which might be of independent interest. 
Before stating it, let us fix some notations that will be used throughout the paper.

\subsection{Notations and a technical result}\label{notTech}

Throughout, we let $T<\infty$ be a positive real number and we consider intervals $[0,T]$. For all $s>0$ and $n\in\mathbb{N}^*$, we write $s_n=\frac{\lfloor ns\rfloor}{n}$, where $\lfloor \cdot\rfloor$ denotes the floor function. We also use $\Gamma(\cdot)$ as the usual Gamma function. For given time indices $0=t_0\leq t_1\leq\ldots\leq t_p$, we denote by $\Sigma((0,t_1),\ldots,(t_{p-1},t_p))$ the covariance matrix of the Gaussian vector $(B_{t_i}-B_{t_{i-1}})_{i\in\llbracket 1,p\rrbracket}$,
where $B$ is a fractional Brownian motion of Hurst index $H\in(0,1)$ (see Section \ref{prl} for more details). Finally, for a positive definite matrix $Q\in\mathbb{R}^{q\times q}$,
we define a function 
\begin{equation}\label{densityExp}
d_Q(x)=e^{-x^TQ^{-1}x}, \quad x\in\mathbb{R}^q.
\end{equation}

\medskip 

Let us choose $p,q\in\mathbb{N}^*$. For every subset $A\subset\llbracket 1,p+q\rrbracket$, we will write $\bar A:=\llbracket 1,p+q\rrbracket\setminus A$,
 and for $i\in\{1,\ldots,Card(A)\}$, $A[i]$ will denote the $i$-th element of $A$ (in the canonical order).
Finally, in what follow, $D$ and $\beta$ will denote generic constants depending only on $T,H,p,q$ and possibly $M$ (see below the statement of hypothesis $(H_1)$ in Proposition \ref{mainBound})
and whose value will potentially change from line to line.

\medskip 

Consider a functional of the increments of a fractional Brownian motion $Y:=F((B_{t_i}-B_{t_{i-1}})_{i\in\llbracket 1,m\rrbracket})$. Our main technical ingredient is the forthcoming Proposition \ref{mainBound} that provides a bound
for the expectation of $\mathbb{E}[Y]$. More importantly, under appropriate growth conditions on $F$ and when some of the time increments $|t_i-t_{i-1}|$ are ''small'' and some others are ''large'', it  provides a quantitative bound on the difference between the expected value of $Y$ and the expected value of the same function evaluated on a Gaussian vector with independent components. Roughly speaking, this means that ''small'' increments are close to being independent from ''large'' increments provided that ''small'' increments are not adjacent to each other. Intuitively, this means that as long as there is at least one large increment in between, small increments are almost independent of each other.

\smallbreak The proof of Proposition \ref{mainBound} relies on a detailed study of the covariance matrix of the vector $(B_{t_i}-B_{t_{i-1}})_{i\in\llbracket 1,m\rrbracket}$, and is postponed into Section \ref{prfMain2}.

\begin{proposition}\label{mainBound}
Let $p,q\in\mathbb{N}^*$ with $p\leq q$ and let $0=t_0<t_1\ldots<t_{p+q}\leq T$. Let $B$ be a fractional Brownian motion of Hurst index $H>\frac{1}{2}$ and let $F:\mathbb{R}^{p+q}\mapsto\mathbb{R}$ be a function
 with at most polynomial growth. Set also the following three hypothesis:

\begin{enumerate}
\item[$(H_1)$] There is a real $M>0$, a function $G:\mathbb{R}^{p+q}\rightarrow\mathbb{R}_+$ with at most polynomial growth, and a subset $J\subset\llbracket 1,p+q\rrbracket$ with cardinality $p$
and $1\in \bar J$  such that, for all $(x_1,\ldots,x_{p+q})$,
\begin{eqnarray}\label{entoglement}
\left|F\left(x_1,x_{1}+x_2,\ldots,\sum_{k=1}^{p+q}x_k\right)\right|\leq G(x_1,\ldots,x_{p+q})\prod_{i=1}^{q}\mathbb{I}_{\left\{|x_{\bar J[i]}|\leq M\sum_{l=1}^p|x_{J[l]}|\right\}}.\notag\\
\end{eqnarray}
\item[$(H_2)$] There exists an $h\in(0,1)$ such that for all $(i,j)\in J\times\bar J$ we have 
$$(t_{i}-t_{i-1})\leq h(t_j-t_{j-1}).$$
\item[$(H_3)$] For all $(i,j)\in J^2$, we have $i\neq j\implies |i-j|\geq 2$.
\end{enumerate}
Then we have:
\bigskip

\begin{enumerate}
\item[(1.)] Under the hypothesis $(H_1)$, there are two constant $\beta,D>0$, independent of the chosen time points $t_i$, such that for all $a\in\mathbb{R}$ we have 
\begin{eqnarray}\label{Basic}
&&\mathbb{E}\left[F(B_{t_1}-a,\ldots, B_{t_{p+q}}-a)\right]\notag\\
&\leq& \frac{\beta e^{-Da^2}}{\prod_{i=1}^{p+q}|t_i-{t_{i-1}}|^H}\int_{\mathbb{R}^{p+q}}\left|F\left(y_1,\ldots, \sum_{k=1}^{p+q}y_k\right)\right|\prod_{i=1}^pe^{-D\frac{y^2_{J[i]}}{|t_{J[i]}-t_{J[i]}-1|^{2H}}}dy_i.\notag\\
\end{eqnarray}
\item[(2.)] Assume in addition that $(H_2)$ and $(H_3)$ also holds true. Define the centered Gaussian vector $X:=(X_1,\ldots,X_p)$ via covariance 
$$Cov(X_i,X_j)=\delta_{ij}\left(t_{J[j]}-t_{J[j]-1}\right)^{2H},$$ and, for $x=(x_1,\ldots,x_{q})$, let $\tilde B(x)=(\tilde B_1,\ldots,\tilde B_{p+q})$ be defined, for all $k\in\llbracket 1,p+q\rrbracket,$ as
\begin{equation*}\tilde B_{k}(x)-\tilde B_{k-1}(x):=\left\{
      \begin{array}{cll}
        x_{i} \textit{ if } k= \bar J[i]\\
        X_{i}-X_{i-1}\textit{ if } k= J[i],\\
      \end{array}
    \right.
    \end{equation*}
    where $\tilde B_0 = 0$. 
If $\Sigma':=\Sigma\left((t_{\bar J[i-1]},t_{\bar J[i]})\right)_{i\in\llbracket 1,p+q\rrbracket}$ denotes the covariance matrix of $(B_{t_{\bar J[i]}}-B_{t_{\bar J[i-1]}})_{i\in\llbracket 1,q\rrbracket}$,
then, for all $a\in\mathbb{R}$, we have
\begin{eqnarray}\label{complete}
&&\left|\mathbb{E}\left[F(B_{t_1}-a,\ldots, B_{t_{p+q}}-a)\right]-\frac{d_{\Sigma'}((a,0,\ldots,0))}{(2\pi)^\frac{q}{2}\sqrt{det(\Sigma')}}\right.\notag\\
&&\left.\times\int_{\mathbb{R}^{q}}\mathbb{E}[F(\tilde B(y))]dy_{1}\ldots dy_q\right|\notag\\
&\leq&\frac{h^{2-2H}\beta e^{-Da^2}}{\prod_{i=1}^{p+q}|t_i-{t_{i-1}}|^H}\int_{\mathbb{R}^{p+q}}\left|F\left(y_1,\ldots, \sum_{k=1}^{p+q}y_k\right)\right|\prod_{i=1}^pe^{-D\frac{y^2_{J[i]}}{|t_{J[i]}-t_{J[i]-1}|^{2H}}}dy_i.\notag\\
\end{eqnarray}
where $d_{\Sigma'}$ is given by \eqref{densityExp}, and $\beta,D$ are as in \eqref{Basic}.
\end{enumerate} 

\end{proposition}
The subset $J$ in condition $(H_1)$ corresponds to the indices of the ''small'' increments, and condition $(H_1)$ allows to upper bound the function $F$ evaluated at points $B_{t_k}$ with a function $G$ depending on the increments. The complement $\overline{J}$ corresponds to the indices of ''large'' increments. Moreover, the product of indicators in $(H_1)$ implies that variables corresponding to large increments are bounded by (a multiple) of small increments, indexed with $J$. As a consequence of this, upper bounds on the right hand sides of statements involving only density for small increments is well-defined. Condition $(H_2)$ simply measures (with a constant $h$) how much larger the large increments are compared to the small ones. Similarly, $(H_3)$ simply means that small increments are not adjacent to each other. Then the more important second statement gives an explicit upper bound for the difference if one compares the true expectation to the one where small increments are considered independent, and dependence is taken account only through the covariance matrix consisting of large increments indexed by $\overline{J}$. 

The rest of the paper is organised as follows. In Section \ref{prl} we recall some useful notions and results related to the fractional Brownian motion, the local time, and to the stochastic integration of discontinuous processes.
 In Section \ref{prfMain1} we prove our main result, Theorem \ref{main}. The proof of our main technical ingredient, Proposition \ref{mainBound}, is presented in  Section \ref{prfMain2}. Finally, some other useful technical lemmas are gathered into the Appendix.

\section{Preliminaries}\label{prl}

\subsection{The fractional Brownian motion}

Let $H\in (0,1)$. The (one dimensional) fractional Brownian motion of Hurst index $H$ is defined as a centered Gaussian process $(B_t)_{t\geq 0}$ with $B_0=0$ and a covariance function $\mathbb{E}[B_sB_t]=\frac{1}{2}\left(t^{2H}+s^{2H}-|t-s|^{2H}\right).$ The fact that the function $f:s,t\rightarrow\frac{1}{2}\left(t^{2H}+s^{2H}-|t-s|^{2H}\right)$ is symmetric and semidefinite is proved in e.g \cite{nourdin2012selected}, yielding that we have a proper covariance function.
In the particular case $H=\frac{1}{2}$, the covariance reduces to $\mathbb{E}[B_sB_t]=s\wedge t$ that corresponds to the covariance of the standard Brownian motion.

\medskip 

We make use of the following well-known elementary properties.

\begin{lemma}\label{bounds}
Let $B$ be a fractional Brownian motion with Hurst parameter $H\in (0,1)$. Then,

\begin{itemize}
\item $B$ is $H$-self similar, i.e. $\forall c>0$, 
$$
(B_{ct})_{t\geq 0}\overset{\mathcal L}{=}(c^HB_{t})_{t\geq 0},
$$
where $\overset{\mathcal L}{=}$ stands for equality in law.
\medskip

\item $B$ has stationary increments, i.e. $$(B_{t+s}-B_{s})_{s\geq 0}\overset{\mathcal L}{=}(B_t)_{t\geq 0}.$$

\medskip

\item There is a modification of $B$ with almost surely $\alpha$-H\"older continuous paths for any $0<\alpha<H$.

\end{itemize}
\end{lemma}

Computations with the fractional Brownian motions are often more involved than for the standard Brownian motion due to the complex structure of the covariance function. Fortunately, they can be simplified with the use of some useful estimates given below. For the proof, see \cite{jaramillo2021approximation}.

\begin{lemma}\label{estimates}
Let $B$ be a fractional Brownian motion with $H>\frac{1}{2}$ and let $T>0$. Then, there exists a constant $\beta>0$, depending solely on $T$ and $H$, such that:
\begin{itemize}
\item for all $0\leq t\leq T$, $0\leq u\leq v\leq T$, $|\mathbb{E}[(B_v-B_u)B_t]|\leq t^{2H-1}|v-u|$.
\item for all $0\leq u\leq v\leq T$ and $0\leq x\leq y\leq T$ with $u\leq x$ and $v-u\leq\frac{|x-v|}{2}$,
 $$|\mathbb{E}[(B_v-B_u)(B_y-B_x)]|\leq \beta(v-u)(y-x)|y-u|^{2H-2}.$$
\end{itemize}
\end{lemma}

\begin{remark} In the case $u=k,v=k+1,x=j$, and $y=j+1$ for $j\leq k$, we recover as a particular case of Lemma \ref{estimates} the formula
\begin{eqnarray*}\mathbb{E}[(B_{j+1}-B_j)(B_{k+1}-B_k)]=\mathbb{E}[B_1(B_{k+1-j}-B_{k-j})] \leq \beta(k-j)^{2H-2}.
\end{eqnarray*}
This should be compared to the known exact asymptotic expression \begin{eqnarray*}\mathbb{E}[(B_{j+1}-B_j)(B_{k+1}-B_k)]=\mathbb{E}[B_1(B_{k+1-j}-B_{k-j})]\\\sim_{k-j\rightarrow\infty}H(2H-1)(k-j)^{2H-2}.\end{eqnarray*}
\end{remark}
In the sequel, we will write, for all $\{(a^1_1,a^1_2),\ldots,(a^m_1,a^m_2)\}\in [0,t]^{2m}$, 
\begin{eqnarray}\label{covariance}
&&\Sigma\left((a^1_1,a^1_2),\ldots,(a^m_1,a^m_2)\right)\\
&:=&\begin{bmatrix}
    \mathbb{E}[(B_{a^1_1}-B_{a^1_2})^2]  & \dots  & \mathbb{E}[(B_{a^1_1}-B_{a^1_2})(B_{a^m_1}-B_{a^m_2})] \\
    \vdots & \vdots & \vdots & \\
    \mathbb{E}[(B_{a^m_1}-B_{a^m_2})(B_{a^1_1}-B_{a^1_2})] &\dots  & \mathbb{E}[(B_{a^m_1}-B_{a^m_2})^2]
\end{bmatrix},\notag
\end{eqnarray}
the covariance matrix of the increments of $B^H$. We will also make use of the following property, known as the \textit{local non-determinism of the fractional Brownian motion}\footnote{The original formulation of the local non-determinism property, introduced in \cite{berman1973local}, is stated in terms of conditional variances. As in the Gaussian case conditional processes are Gaussian as well, one can deduce the equivalent formulation presented here in the case of the fractional Brownian motion, see \cite[Lemma 2.3]{berman1973local}.}:
 for all $m\in\mathbb{N}^*$ there exists a constant $L_{H,m}>0$ such that for all $0=s_0<s_1\leq\ldots\leq s_m$ and for all $(u_1,\ldots,u_m)$,
\begin{equation}\label{localNd}Var\left(\sum_{i=1}^mu_i(B_{s_i}-B_{s_{i-1}})\right)\geq L_{H,m}\sum_{i=1}^m|u_i|^2(s_i-s_{i-1})^{2H}.\end{equation}
By using the H\"older inequality, it is easy to establish the converse bound
\begin{equation}\label{VarLow}
Var\left(\sum_{i=1}^mu_i(B_{s_i}-B_{s_{i-1}})\right)\leq m\sum_{i=1}^m|u_i|^2(s_i-s_{i-1})^{2H}.
\end{equation}
Consequently, local non-determinism means that the variance of the linear combination of the increments behaves, up to constants, as the linear combination of the variances of the increments. More details about the local determinism property of Gaussian processes can be found in \cite{berman1973local} and \cite{xiao2006properties}. 
\smallbreak

The non-determinism implies the following property that will be used in the proof of Proposition \ref{mainBound}. 
Let $0\leq s_0<s_1\leq s_2<s_3\ldots\leq s_{2m-2}<a_{2m-1}$ and let $\Sigma((s_{2i+1},s_{2i})_{i\in\llbracket 0,m-1\rrbracket})$ be as in \eqref{covariance}.
 Then, there is a constant $k_{H,m}>0$ (depending only on $H,T,m$) such that

\begin{equation}\label{Sandwich}
k_{H,m}\prod_{i=0}^{m-1}|s_{2i+1}-s_{2i}|^{2H}\leq det\left(\Sigma\left((s_{2i+1},s_{2i})_{i\in\llbracket 0,m-1\rrbracket}\right)\right).
\end{equation}

We similarly have the upper bound

\begin{equation}\label{Sandwich2}
det\left(\Sigma\left((s_{2i+1},s_{2i})_{i\in\llbracket 0,m-1\rrbracket}\right)\right)\leq m!\prod_{i=0}^{m-1}|s_{2i+1}-s_{2i}|^{2H}.
\end{equation}

As a consequence of \eqref{localNd} and \eqref{VarLow}, the eigenvalues $\lambda_1,\ldots,\lambda_m$ of the matrix $\Sigma\left((s_i,s_{i-1})_{i\in\llbracket 1,m\rrbracket}\right)$ verifies

\begin{equation}\label{eigenvalues}\forall i\in\{1,\ldots,m\}, ~L_{H,m}\min_{j\in\{1,\ldots,m\}}|s_j-s_{j-1}|^{2H}\leq\lambda_i\leq m\max_{j\in\{1,\ldots,m\}}|s_j-s_{j-1}|^{2H}.
\end{equation}
When there is no ambiguity, we will drop the dependency on $m$ and simply write $k_H$ and $L_H$.

\subsection{Local times}
\smallbreak
Given a generic $d$-dimensional stochastic process $(X_t)_{t\geq 0}$, we define the occupation measure (on an interval $I$) as 
$$
\eta_{X}(B) = Leb(s\in I:X_s \in B).
$$
That is, $\mu_X$ measures the ''time'' the process $X$ spends on a set $B$. If the occupation measure $\eta$ is absolutely continuous with respect to the Lebesgue measure, the density $L_I(a)$ of $\eta_X$ with respect to the Lebesgue measure is called \emph{the local time} of $X$. In what follows, we write simply $L_t(a) = L_{[0,t]}(a)$, for a given $t>0$. From the very definition of the local time, we have the occupation times formula
$$
\int_0^t f(X_s)ds = \int_{-\infty}^\infty f(y)L_t(y)dy.
$$
Consequently, we have 
$$L_t(a):=\lim_{\epsilon\rightarrow 0}\frac{1}{2^d\epsilon^d}\int_0^t\mathbb{I}_{B_s\in C_{\epsilon}(a)}ds,$$
where $C_{\epsilon}(a)=[a-\epsilon,a+\epsilon]^d$, and the limit exists for Lebesgue almost every $a$, or for every $a$ provided that the local time $L_t(a)$ is a continuous function (almost surely). In the case of the $d$-dimensional fractional Brownian motion,
 it is well known that the local time exists and is continuous for any Hurst index $H\in(0,1)$ such that $Hd<1$, see \cite{geman1980occupation}.

\medskip 

Let now $d=1$. For $u,v>0$ and $x,y\in\mathbb{R},$ let $\phi_{u,v}(x,y)$ be the density of $(B_u,B_v)$ at $(x,y)$, that is
\begin{equation}
    \label{density}
    \phi_{u,v}(x,y):=\frac{1}{2\pi\sqrt{det(\Sigma_{u,v})}}e^{-\frac{v^{2H}x^2}{2det(\Sigma_{u,v})}-\frac{u^{2H}y^2}{2det(\Sigma_{u,v})}+\frac{xy\mathbb{E}[B_uB_v]}{det(\Sigma_{u,v})}},
\end{equation}
where $\Sigma_{u,v}$ denotes the covariance matrix of the vector $(B_u,B_v)$. More generally, let $\phi_{u_1,\ldots,u_p}(x_1,\ldots,x_p)$ be the density of the $(B_{u_1},\ldots, B_{u_p})$ at $(x_1,\ldots,x_p)$. Then we have the following elementary result
regarding the approximation of the fractional Brownian motion local time in the $L^p(\Omega)$ norm.

\begin{lemma}\label{ApproxLocalTime}Let $H\in (0,1)$, $t\geq0$ and $L_t(a)$ the local time of the fractional Brownian motion of index $H$. Then, for all $p>0$,
$$L^{\epsilon}_t(a):=\frac{1}{2\epsilon}\int_0^t\mathbb{I}_{B_s\in [a-\epsilon,a+\epsilon]}ds\overset{L^p}{\underset{\epsilon\rightarrow 0}\longrightarrow} L_t(a).$$
\end{lemma}

\begin{proof}
Since $L^{\epsilon}_t$ converges in probability to $L_t$ for all $t>0$, it is sufficient to prove that $L^{\epsilon}_t$ is uniformly integrable in every space $L^p(\Omega)$. For this on the other hand, it is enough to show that $\forall p\in\mathbb{N}^*, \sup_{0<\epsilon<1}\mathbb{E}[|L^{\epsilon}_t|^p]<\infty$.

\medskip 

For $p\in\mathbb{N}^*$, we have 
\begin{eqnarray*}
\mathbb{E}[|L^{\epsilon}(a)_t|^p]&=&\frac{1}{2^p\epsilon^p}\int_{[0,t]^p}\prod_{i=1}^p\mathbb{E}[\mathbb{I}_{B_{u_i}\in [a-\epsilon,a+\epsilon]}]du_1\ldots du_p\\
&=&\frac{1}{2^p\epsilon^p}\int_{[0,t]^p}\int_{[a-\epsilon,a+\epsilon]^p}dx_1\ldots dx_pdu_1\ldots du_p\\
&&\times\frac{e^{-\frac{1}{2}x^T\Sigma^{-1}((0,u_1),\ldots,(0,u_p))x}}{(2\pi)^\frac{p}{2}\sqrt{det\left(\Sigma((0,u_1),\ldots,(0,u_p))\right)}}\\
&\leq &K^p_H\int_{[0,t]^p}\frac{1}{(2\pi)^\frac{p}{2}\sqrt{det\left(\Sigma((0,u_1),\ldots,(0,u_p))\right)}}du_1\ldots du_p\\
&\leq&p!\int_{0=u_0< u_1<\ldots<u_p< t=u_{p+1}}\frac{1}{(2\pi)^\frac{p}{2}}\prod_{i=1}^{p+1}(u_i-u_{i-1})^{-H}du_1\ldots du_p\\
&<&+\infty,
\end{eqnarray*}
where the penultimate line comes from the local non-determinism property of the fractional Brownian motion. This conclude the proof.
\end{proof}
As a result of the previous computations, we recover the well-known formula:
\begin{equation}\label{squareLocalTime}\forall t\geq 0,\mathbb{E}[(L_t(a))^p]=\int_0^t\ldots\int_0^t\phi_{u_1,\ldots,u_p}(a,\ldots,a)du_1\ldots du_p.\end{equation}

\subsection{Riemann-Stieltjes integration}
When $x,y:\mathbb{R}\rightarrow\mathbb{R}$ are two $\alpha$ (resp $\beta$)-H\"older continuous functions with $\alpha+\beta>1$ and $0\leq a\leq b$,
 it is well-known already since the early 1900's (see \cite{young1936inequality}) that one can define the integral $\int_a^b xdy$ as
$$\int_a^bx_sdy_s:=\lim_{n\rightarrow\infty}\sum_{k=\lfloor na\rfloor}^{\lfloor nb\rfloor-1}x_\frac{k}{n}\left(y_\frac{k+1}{n}-y_{\frac{k}{n}}\right),$$
and this can be slightly extended to functions with $\frac{1}{\alpha}$ (resp $\frac{1}{\beta}$)-finite variations. However, this cannot be applied in a straightforward manner in the case $x_s=F(u_s)$, where $u$ is not Lipschitz continuous and $F$ has discontinuities.
Indeed, then the paths of $x$ can easily have infinite $p$-variation for every $p>0$ as can be seen already the simple case $F(\cdot) = I_{\cdot > a}$ and $u$ being a path of, e.g. a fractional Brownian motion. In such cases, the integral can be defined by using either Z\"ahle type arguments (see \cite{chen2019pathwise})
 or an extension of the sewing Lemma (see \cite{yaskov2019pathwise}), provided that the path $u$ satisfies so-called sufficient variability condition (see \cite{chen2019pathwise} and later extensions \cite{HTV1,HTV2}). We will adopt the first approach based on fractional calculus that is more convenient for explicit computations. For details, we refer to the seminal paper \cite{zahle1998integration}.

Let $a<b$ and let $f\in L^1(\mathbb{R})$. The fractional integrals of order $\alpha>0$ of $f$ are defined as
\begin{eqnarray*}
&&I^\alpha_{a+}f(t):=\frac{1}{\Gamma(\alpha)}\int_a^t\frac{f(s)}{(t-s)^{1-\alpha}}ds,\\
&&I^\alpha_{b-}f(t):=\frac{(-1)^{\alpha}}{\Gamma(\alpha)}\int_t^b\frac{f(s)}{(t-s)^{1-\alpha}}ds.
\end{eqnarray*}
The above integrals converge for almost all $t\in (a,b)$. Furthermore, the operator $I_{a+}$ and $I_{b-}$ defines an injective linear mapping from $L^1$ to $L^1$.
 For $\alpha\in (0,1)$ and $f\in I^\alpha_{a+}(L^1)$, $g\in I^{\alpha}_{b-}(L^1)$, one can define the \textit{Weyl-Marchaud derivatives} as
\begin{eqnarray*}
&&D^{\alpha}_{a+}f(t):=\frac{1}{\Gamma(1-\alpha)}\left(\frac{f(t)}{(t-a)^\alpha}+\int_a^t\frac{f(t)-f(s)}{(t-s)^{\alpha+1}}ds\right)\\
&&D^{\alpha}_{b-}f(t):=\frac{(-1)^{\alpha}}{\Gamma(1-\alpha)}\left(\frac{f(t)}{(t-b)^\alpha}+\int_t^b\frac{f(t)-f(s)}{(t-s)^{\alpha+1}}ds\right).
\end{eqnarray*}

The following result on the existence of Riemann-Stieltjes integrals can be found in \cite{chen2019pathwise}:

\begin{theorem}\label{DefIntegral}

Let $X,Y$ be two real valued stochastic processes satisfying the following conditions:

\begin{enumerate}
\item $X$ (resp $Y$) is $\alpha$ (resp $\beta$)-H\"older continuous, with $\alpha+\beta>1$.
\item For almost every $t\in (0,T)$, $X_t$ has a density $p_t$ satisfying 
$$\int_0^T(\sup_{x\in\mathbb{R}}p_t(x))dt<\infty.$$
\item There exists $H>0$ and $r>\frac{1}{H}$ such that
$$\frac{\mathbb{E}[|X_t-X_s|^r]^\frac{1}{r}}{|t-s|^H}<\infty.$$
\end{enumerate}

Then, if $F$ is a real valued function of locally bounded variation, the Riemann sum 
$$\sum_{k=\lfloor na\rfloor}^{\lfloor nb\rfloor}F\left(X_\frac{k}{n}\right)\left(Y_{\frac{k+1}{n}\wedge b}-Y_\frac{k}{n}\right)$$
converges in probability for all $t\in (0,T)$ to a limit, denoted by $\int_a^bF(X_s)dY_s$. In addition, by setting $f_s=F(X_s)$ and assuming $\alpha>1-\beta$, we have for any $\gamma\in (1-\beta,\alpha)$ that
\begin{eqnarray*}
&&\int_a^bf_tdY_s\\
&=&(-1)^{\alpha}\int_a^bD^{\gamma}_{a+}\left(f-f_{a+}\right)(s)D^{1-\gamma}_{b-}\left(Y-Y_{b}\right)(s)ds+f_{a+}\left(Y_b-Y_a\right).
\end{eqnarray*}
Furthermore, if $F$ is a Lipschitz function whose derivative is of locally bounded variations, one has, almost surely,

\begin{equation}\label{ChangeVar}F(x_t)=F(x_0)+\int_0^tF'(x_s)dx_s.\end{equation}
\end{theorem}
Notice that if $X=B^1$ and $Y=B^2$ are two fractional Brownian motions with Hurst indices $H_1$ and $H_2$ with $H_1+H_2>1$, the assumptions of Theorem \ref{DefIntegral}
 are satisfied for $\alpha=H_1$, $\beta=H_2$ and $p_t:x\rightarrow \frac{1}{\sqrt{2\pi}t^H}e^{-\frac{x^2}{2t^{2H}}}$. Consequently, the Riemann-Stieltjes type integrals in Theorem \ref{prfMain1} exists and the problem studied in the present paper is well-posed.

\section{Proof of Theorem \ref{main}}\label{prfMain1}
The present section is dedicated to the proof of the Theorem \ref{main}. The proof is divided into three steps. 
In the first two steps, we prove the result in the case where $f(x)=\mathbb{I}_{x>a}$ for some $a\in\mathbb{R}$, first for $i\neq j$ and then for $i=j$. As such, these two steps prove Proposition \ref{particular} concerning this particular case. In the final step, we apply this result to cover general functions $f$. Throughout the proof, the unimportant constants $\beta$ and $D$ may change from line to line.

\medskip

\subsection{Step 1: Proof of Theorem \ref{main} for $i\neq j$ and $f:x\rightarrow\mathbb{I}_{x>a}$.}
Let $\gamma\in \left(\frac{1}{2},H\right)$ and recall that $s^k_n=\frac{\lfloor ns^k\rfloor}{n}$, for $k\in \{1,2\}$. Thanks to the Theorem \ref{DefIntegral}, we have \footnote{By fractional integration by parts formula, one obtains that the roles of upper and lower bounds $a$ and $b$ in $f$ and $g$ can be interchanged which we have applied here in order to simplify few steps in our computations.}
\begin{eqnarray}\label{Part1}&&\mathbb{E}[\left|S_n^{i,j}(t)\right|^2]\\&=&\beta n^{2(2H-1)}\int_{[0,t]^2}\notag\\
&&\times\mathbb{E}\left[\prod_{k=1}^2\left(\frac{\mathbb{I}_{B^i_{\left(s^k_n+\frac{1}{n}\right)\wedge t>a}}-\mathbb{I}_{B^i_{s^k}>a}}{(\left(s^k_n+\frac{1}{n}\right)\wedge t-s^k)^{1-\gamma}}
    +\int_{s^k}^{\left(s^k_n+\frac{1}{n}\right)\wedge t}\frac{\mathbb{I}_{B^i_{y_k}>a}-\mathbb{I}_{B^i_{s^k}>a}}{(y_k-s^k)^{2-\gamma}}dy_k\right)\right]\notag\\
&&\times\mathbb{E}\left[\prod_{k=1}^2\left(\frac{B^j_{s^k}-B^j_{s^k_n}}{\left(s^k-s^k_n\right)^{\gamma}}+\int_{s^k_n}^{s^k}\frac{B^j_{s^k}-B^j_{y_k}}{(y_k-s^k)^{1+\gamma}}dy_k\right)\right]ds^1ds^2.\notag
\end{eqnarray}
By symmetry, it suffices to consider $0<s^1< s^2$ which is what we will do in the sequel. Then, for $y_1\in [s^1_n,s^1]$ and $y_2\in [s^2_n,s^2]$, we have, thanks to Lemma \ref{estimates} (when $s^2_n-s^1\geq\frac{2}{n}$) and the H\"older inequality (when $|s^2_n-s^1|\leq\frac{2}{n}$), that

\begin{eqnarray}\label{low1}
\left|\mathbb{E}\left[\prod_{k=1}^2\left(B^j_{y_k}-B^j_{s^k}\right)\right]\right|
      \leq \left(\beta\frac{\left|y_2-s^2\right|\left|y_1-s^1\right|}{|s^2_n-s^1|^{2-2H}}\wedge\left|y_2-s^2\right|^H\left|y_1-s^1\right|^H \right).\notag\\
\end{eqnarray}

\medskip

Let 
$$F:(v,w,x,y)\rightarrow\left(\mathbb{I}_{w\geq 0}-\mathbb{I}_{v\geq 0}\right)\left(\mathbb{I}_{y\geq 0}-\mathbb{I}_{x\geq 0}\right).$$
Then $F$ satisfies hypothesis $(H_1)$ in Proposition \ref{mainBound} with $p=q=2$, $J=\{2,4\}$, $M=1$ and $G\equiv 1$. 
Hence we may apply \eqref{Basic} to see that, for all $0<s^1< s^2$, for all $y_1\in [s^1,\left(s^1_n+\frac{1}{n}\right)\wedge t]$ and $y_2\in [s^2,\left(s^2_n+\frac{1}{n}\right)\wedge t],$ we have 

\begin{eqnarray*}
&&\mathbb{E}\left[\prod_{k=1}^2\left|\mathbb{I}_{B^i_{y_k}>a}-\mathbb{I}_{B^i_{s^k}>a}\right|\right]\\
&\leq& \frac{\beta e^{-Da^2}}{\prod_{i=1}^4|t_i-t_{i-1}|^H}\int_{\mathbb{R}^4}F\left(x_1,\ldots,\sum_{i=1}^4x_i\right)e^{-\frac{Dx_2^2}{|y_1-s^1|^{2H}}-\frac{Dx_4^2}{|y_2-s^2|^{2H}}}\prod_{i=1}^4dx_i,\\
\end{eqnarray*}
with $t_1=s^1$, $t_2=y_1$, $t_3=s^2$, $t_4=y_2$. Using Lemma \ref{calculus2}, we then have
\begin{equation}\label{low2}
\mathbb{E}\left[\prod_{k=1}^2\left|\mathbb{I}_{B^i_{y_k}>a}-\mathbb{I}_{B^i_{s^k}>a}\right|\right]\leq \beta e^{-Da^2}\frac{|y_2-s^2|^H|y_1-s^1|^H}{|s^1|^H|s^2-y_1|^H}.
\end{equation}
On the other hand, for any $s^1\leq s^2$, we also have, using again Lemma \ref{calculus2}, that
\begin{eqnarray}\label{low3}
\mathbb{E}\left[\prod_{k=1}^2\left|\mathbb{I}_{B^i_{y_k}>a}-\mathbb{I}_{B^i_{s^k}>a}\right|\right]&\leq&\mathbb{E}[|\mathbb{I}_{B^i_{y_2}>a}-\mathbb{I}_{B^i_{s^2}>a}|]\leq \beta e^{-Da^2}\frac{|y_2-s^2|^H}{(s^2)^H}.
\end{eqnarray}
As a consequence of \eqref{low2} and \eqref{low3}, we have, for $s^1<s^2$,

\begin{eqnarray}\label{low4}
&&\int_{s^1}^{s^1_n+\frac{1}{n}\wedge t}\int_{s^2}^{s^2_n+\frac{1}{n}\wedge t}\prod_{k=1}^2\mathbb{E}\left[\frac{|\mathbb{I}_{B_{y_k>a}}-\mathbb{I}_{s^k>a}|}{|y_k-s^k|^{2-\gamma}}\right]dy_2dy_1\notag\\
&=&\int_{s^1}^{s^1_n+\frac{1}{n}\wedge t}\int_{s^2}^{s^2_n+\frac{1}{n}\wedge t}\prod_{k=1}^2\mathbb{E}\left[\frac{|\mathbb{I}_{B_{y_k>a}}-\mathbb{I}_{s^k>a}|}{|y_k-s^k|^{2-\gamma}}\right]dy_2dy_1\mathbb{I}_{\left(s^1_n+\frac{1}{n}\right)\wedge t<s^2-\frac{1}{n}}\notag\\
&&+\int_{s^1}^{s^1_n+\frac{1}{n}\wedge t}\int_{s^2}^{s^2_n+\frac{1}{n}\wedge t}\prod_{k=1}^2\mathbb{E}\left[\frac{|\mathbb{I}_{B_{y_k>a}}-\mathbb{I}_{s^k>a}|}{|y_k-s^k|^{2-\gamma}}\right]dy_2dy_1\mathbb{I}_{\left(s^1_n+\frac{1}{n}\right)\wedge t\geq s^2-\frac{1}{n}}\notag\\
&\leq&\beta e^{-Da^2}\int_{s^1}^{s^1_n+\frac{1}{n}\wedge t}\int_{s^2}^{s^2_n+\frac{1}{n}\wedge t}\frac{|y_2-s^2|^H|y_1-s^1|^H}{|s^1|^H|s^2-y_1|^H}dy_2dy_1\mathbb{I}_{\left(s^1_n+\frac{1}{n}\right)\wedge t<s^2-\frac{1}{n}}\notag\\
&&+\int_{s^1}^{s^1_n+\frac{1}{n}\wedge t}\int_{s^2}^{s^2_n+\frac{1}{n}\wedge t}\prod_{k=1}^2\mathbb{E}\left[\frac{|\mathbb{I}_{B_{y_k>a}}-\mathbb{I}_{s^k>a}|}{|y_k-s^k|^{2-\gamma}}\right]dy_2dy_1\mathbb{I}_{\left(s^1_n+\frac{1}{n}\right)\wedge t\geq s^2-\frac{1}{n}}\notag\\
&\leq& \beta e^{-Da^2}\frac{1}{(s^1)^H\left(s^2-\left(s^1_n+\frac{1}{n}\right)\wedge t\right)^H}\frac{1}{n^{2\gamma+2H-2}}\mathbb{I}_{\left(s^1_n+\frac{1}{n}\right)\wedge t<s^2-\frac{1}{n}}\notag\\
&&+\int_{s^1}^{s^1+\frac{s^2-s^1}{2}}\int_{s^2}^{\left(s^2_n+\frac{1}{n}\right)\wedge t}\prod_{k=1}^2\mathbb{E}\left[\frac{|\mathbb{I}_{B_{y_k>a}}-\mathbb{I}_{s^k>a}|}{|y_k-s^k|^{2-\gamma}}\right]dy_2dy_1\mathbb{I}_{\left(s^1_n+\frac{1}{n}\right)\wedge t\geq s^2-\frac{1}{n}}\notag\\
&&+\int_{s^1+\frac{s^2-s^1}{2}}^{\left(s^1_n+\frac{1}{n}\right)\wedge t}\frac{1}{|y_1-s^1|^{2-\gamma}}dy_1\int_{s^2}^{\left(s^2_n+\frac{1}{n}\right)\wedge t}\mathbb{E}\left[\frac{|\mathbb{I}_{B_{y_2>a}}-\mathbb{I}_{s^2>a}|}{|y_2-s^2|^{2-\gamma}}\right]dy_2\notag\\&&\times\mathbb{I}_{\left(s^1_n+\frac{1}{n}\right)\wedge t\geq s^2-\frac{1}{n}}\notag\\
&\leq&\beta e^{-Da^2}\frac{1}{(s^1)^H\left(s^2-\left(s^1_n+\frac{1}{n}\right)\wedge t\right)^H}\frac{1}{n^{2\gamma+2H-2}}\mathbb{I}_{\left(s^1_n+\frac{1}{n}\right)\wedge t<s^2-\frac{1}{n}}\notag\\
&&+\beta e^{-Da^2}\frac{1}{(s^1)^H\left(\frac{s^2-s^1}{2}\right)^H}\frac{1}{n^{2\gamma+2H-2}}\mathbb{I}_{\left(s^1_n+\frac{1}{n}\right)\wedge t\geq s^2-\frac{1}{n}}\notag\\
&&+\beta e^{-Da^2}\frac{1}{(s^2)^H\left(\frac{s^2-s^1}{2}\right)^{1-\gamma}}\frac{1}{n^{\gamma+H-1}}\mathbb{I}_{\left(s^1_n+\frac{1}{n}\right)\wedge t\geq s^2-\frac{1}{n}}.
\end{eqnarray}

On the other hand, estimating the boundary and cross product terms leads to:
\begin{eqnarray}\label{low5}
  &&\left|\prod_{k=1}^2\mathbb{E}\left[\frac{\mathbb{I}_{B^i_{\left(s^k_n+\frac{1}{n}\right)\wedge t>a}}-\mathbb{I}_{B^i_{s^k}>a}}{(\left(s^k_n+\frac{1}{n}\right)\wedge t-s^k)^{1-\gamma}}\right]\right|\notag\\
  &\leq& \beta e^{-Da^2}\frac{n^{2-2H-2\gamma}}{(s^1)^H(s^2-(s^1_n+\frac{1}{n})\wedge t)^H}\mathbb{I}_{\left(s^1_n+\frac{1}{n}\right)\wedge t<s^2-\frac{1}{n}},
  \end{eqnarray}
and
\begin{eqnarray}\label{low6}
&&\left|\mathbb{E}\left[\frac{\mathbb{I}_{B^i_{\left(s^1_n+\frac{1}{n}\right)\wedge t>a}}-\mathbb{I}_{B^i_{s^1}>a}}{(\left(s^1_n+\frac{1}{n}\right)\wedge t-s^1)^{1-\gamma}}\int_{s^2}^{\left(s^2_n+\frac{1}{n}\right)\wedge t}\frac{|\mathbb{I}_{B^i_{y_2}>a}-\mathbb{I}_{B^i_{s^2}>a}|}{|y_2-s^2|^{2-\gamma}}dy_2\right]\right|\notag\\
&\leq&\frac{\beta e^{-Da^2}n^{2-2H-2\gamma}}{(s^1)^H(s^2-(s^1_n+\frac{1}{n})\wedge t)^H}\mathbb{I}_{\left(s^1_n+\frac{1}{n}\right)\wedge t< s^2-\frac{1}{n}}\notag\\
&&+\frac{\beta e^{-Da^2}n^{1-\gamma-H}}{((s^1_n+\frac{1}{n})\wedge t-s^1)^{1-\gamma}(s^2)^H}\mathbb{I}_{\left(s^1_n+\frac{1}{n}\right)\wedge t\geq s^2-\frac{1}{n}}.
\end{eqnarray}

Plugging \eqref{low1}, \eqref{low4}, \eqref{low5} and \eqref{low6} into \eqref{Part1}, and noticing that when $\left(s^1_n+\frac{1}{n}\right)\wedge t<s^2-\frac{1}{n}$
 we have $s^2-\left(s^1_n+\frac{1}{n}\right)\wedge t\leq \frac{s^2-s^1}{2}$, we can write 

\begin{eqnarray*}
\mathbb{E}[\left|S_n^{i,j}(t)\right|^2]\leq 2\left(R^{1,n}+R^{2,n}\right),
\end{eqnarray*}

with 

\begin{eqnarray*}
R^{1,n}&:=&n^{4H-2}\beta e^{-Da^2}\int_0^{t-\frac{2}{n}}\int_{s^1}^{s^1+\frac{2}{n}}\left(\frac{n^{2-2H-2\gamma}}{(s^1)^H|s^2-s^1|^H}\right.\\
&&\left.+\left(\frac{1}{(s^2)^H|s^2-s^1|^{1-\gamma}}+\frac{1}{(s^2)^H((s^1_n+\frac{1}{n})\wedge t-s^1)^{1-\gamma}}\right.\right.\\
&&\left.\left.+\frac{1}{(s^1)^H((s^2_n+\frac{1}{n})\wedge t-s^2)^{1-\gamma}}\right)n^{1-\gamma-H}\right)n^{2\gamma-2H}ds^2ds^1\\
&\leq&\beta e^{-Da^2} n^{1-H}.
\end{eqnarray*}

\begin{eqnarray*}
R^{2,n}&:=&n^{4H-2}\beta e^{-Da^2}\int_0^t\int_{s^1+\frac{2}{n}}^t\left(\frac{n^{2-2H-2\gamma}}{(s^1)^H|s^2-s^1|^H}\right)\frac{n^{2\gamma-2}}{|s^2-s^1|^{2-2H}}ds^2ds^1\\
&\leq&\beta n^{2H-2}\int_0^t\int_{s^1+\frac{2}{n}}^t\frac{1}{(s^1)^H|s^2-s^1|^{2H-2}}ds^2ds^1\\
&\leq& \beta e^{-Da^2} n^{1-H}.
\end{eqnarray*}

This concludes the proof of the first step.

\subsection{Step 2: Proof of Theorem \ref{main} for $i=j$ and $f:x\rightarrow\mathbb{I}_{x>a}$.}
For $i=j$, we drop the superscripts in the notation and use the change of variable formula \eqref{ChangeVar} to write

\begin{eqnarray*}S_n(t)&=&n^{2H-1}\left(\int_0^t\mathbb{I}_{B_s>a}dB_s-\sum_{k=1}^{\lfloor nt\rfloor}\mathbb{I}_{B_{\frac{k}{n}}>a}\left(B_{\frac{k+1}{n}\wedge t}-B_\frac{k}{n}\right)\right)\\
&=&n^{2H-1}\sum_{k=0}^{\lfloor nt\rfloor}\left(\left|B_{\frac{k+1}{n}\wedge t}-a\right|-\left|B_\frac{k}{n}-a\right|-\mathbb{I}_{B_{\frac{k}{n}}>a}\left(B_{\frac{k+1}{n}\wedge t}-B_\frac{k}{n}\right)\right)\\
&=&n^{2H-1}\sum_{k=0}^{\lfloor nt\rfloor}\left|B_{\frac{k+1}{n}\wedge t}-a\right|\mathbb{I}_{\left\{sgn\left(B_{\frac{k+1}{n}\wedge t}-a\right)sgn\left(B_\frac{k}{n}-a\right)<0\right\}}.
\end{eqnarray*}

To prove the claim, it is enough to establish the two following bounds:

\begin{equation}\label{mixed}
\left|\mathbb{E}[S_n(t)L_t(a)]-\frac{1}{2}\mathbb{E}[L_t(a)^2]\right|\leq Ce^{-Da^2}n^{1-H},
\end{equation}

\begin{equation}\label{squared}
\left|\mathbb{E}[S_n(t)^2]-\frac{1}{4}\mathbb{E}[L_t(a)^2]\right|\leq C e^{-Da^2}n^{1-H}.
\end{equation}
We will proceed as in Step 1, with the help of Proposition \ref{mainBound}. 
We will first show the bound \eqref{mixed}, and then briefly sketch the proof for the bound \ref{squared} that can be handled similarly.

\begin{proof}[Proof\nopunct]\textit{of \eqref{mixed}:}
In order to obtain hypothesis $(H_1)$, we introduce a modified version 
\begin{eqnarray*}
&&\left(S^{\epsilon}_n(t)\right)[k]:=\left|B_{\frac{k+1}{n}\wedge t}-a\right|\mathbb{I}_{\left\{sgn\left(B_{\frac{k+1}{n}\wedge t}-a\right)sgn\left(B_\frac{k}{n}-a\right)<0\right\}}\mathbb{I}_{\left|B_{\frac{k+1}{n}\wedge t}-a\right|>\epsilon}.
\end{eqnarray*}
Now, according to Lemma \ref{ApproxLocalTime} and the H\"older inequality, we have that

\begin{eqnarray*}\mathbb{E}[S_n(t)L_t(a)]
&=&\lim_{\epsilon\rightarrow 0}\mathbb{E}[S_n(t)L^{\epsilon}_t(a)]\\
&=&\lim_{\epsilon\rightarrow 0}\mathbb{E}[S^{\epsilon}_n(t)L^{\epsilon}_t(a)]\\
&=&\lim_{\epsilon\rightarrow 0} n^{2H-1}\sum_{k=0}^{\lfloor nt\rfloor}\int_0^tA^{\epsilon}_{s,k/n}ds\\
&=&\lim_{\epsilon\rightarrow 0} n^{2H}\int_{0}^t\int_0^tA^{\epsilon}_{u,v_n}du dv,
\end{eqnarray*}
with $A^{\epsilon}_{s,k/n}=\frac{1}{2\epsilon}\mathbb{E}[\left(S^{\epsilon}_n(t)\right)[k]\mathbb{I}_{B_s\in [a-\epsilon,a+\epsilon]}]$. Recall that by \eqref{squareLocalTime}, we have
$$
\mathbb{E}[L_t(a)^2] = \int_0^t \int_0^t \phi_{u,v}(a,a)dudv.
$$
This leads to the bound 
$$
\left|\mathbb{E}[S_n(t)L_t(a)]-\frac{1}{2}\mathbb{E}[L_t(a)^2]\right|\leq R^{1,n}+R^{2,n}+R^{3,n}+R^{4,n},
$$
with
\begin{eqnarray*}
R^{1,n}&:=&\lim\sup_{\epsilon\rightarrow 0}n^{2H}\int_{0}^t\int_0^t\left|A^{\epsilon}_{u,v_n}-\frac{\phi_{u,v}(a,a)}{2n^{2H}}\right|\\
&&\times\mathbb{I}_{min\left(u,v_n,|u-v_n|\right)>\frac{2}{n}}dudv\\
R^{2,n}&=&n^{2H}\int_{0}^\frac{2}{n}\int_\frac{2}{n}^t\left(\sup_{\epsilon\in (0,1)}A^{\epsilon}_{u,v_n}+\frac{\phi_{u,v}(a,a)}{2n^{2H}}\right)dudv\\
R^{3,n}&=&n^{2H}\int_{\frac{2}{n}}^t\int_{0}^\frac{1}{n}\left(\sup_{\epsilon\in (0,1)}A^{\epsilon}_{u,v_n}+\frac{\phi_{u,v}(a,a)}{2n^{2H}}\right)dudv\\
R^{4,n}&=&n^{2H}\int_{0}^t\int_{u-\frac{2}{n}\vee 0}^{u+\frac{1}{n}\wedge t}\left(\sup_{\epsilon\in (0,1)}A^{\epsilon}_{u,v_n}+\frac{\phi_{u,v}(a,a)}{2n^{2H}}\right)dudv.\\
\end{eqnarray*}
We prove that we have 
\begin{eqnarray*}
&&\lim\sup_{\epsilon\rightarrow 0}\left|A^{\epsilon}_{s,\frac{k}{n}}-\frac{\left(\left(\frac{k}{n}+\frac{1}{n}\right)\wedge t-\frac{k}{n}\right)^{2H}\phi_{\frac{k}{n},s}(a,a)}{2}\right|\notag\\
&\leq&\beta h^{2-2H}e^{-Da^2}\frac{1}{\left(\frac{k}{n}\right)^H\left|s-\frac{k}{n}\right|^Hn^{2H}}
\end{eqnarray*}
for $\min\left(s,\frac{k}{n},\left|s-\frac{k}{n}\right|\right)\geq \frac{2}{n}$ that allows to handle $R^{1,n}$, and
\begin{eqnarray*}
A^{\epsilon}_{s,\frac{k}{n}}&\leq&\frac{\beta e^{-Da^2}}{\min\left(s,\frac{k}{n}\right)^H\left|s-\frac{k}{n}\right|^H\frac{1}{\left(\frac{k+1}{n}\wedge t-\frac{k}{n}\right)^H}}\\
&&\times\int_{\mathbb{R}^2}|x+y|\mathbb{I}_{sgn(x)\neq sgn(y)}e^{-D\frac{\left(\frac{k+1}{n}\wedge t-\frac{k}{n}\right)^{2H}x^2}{2}}dxdy\\
&\leq&\beta\frac{n^{-2H}}{min\left(s,\frac{k}{n}\right)^H\left|s-\frac{k}{n}\right|^H}e^{-Da^2}
\end{eqnarray*}
for $\min\left(s,\frac{k}{n},\left|s-\frac{k}{n}\right|\right)<\frac{2}{n}$ that allows to handle $R^{2,n}-R^{4,n}$. 
\begin{itemize}
\item Let first $min\left(s,\frac{k}{n}\left|s-\frac{k}{n}\right|\right)>\frac{2}{n}$. We define a function $F^{\epsilon}:\mathbb{R}^3\mapsto\mathbb{R}$ defined as
\begin{equation*}F^\epsilon(x):=\left\{
      \begin{array}{cll}
        &&\frac{1}{2\epsilon}\mathbb{I}_{x_1\in [-\epsilon,\epsilon]}|x_3|\mathbb{I}_{sgn(x_2)\neq sgn(x_3)}\mathbb{I}_{|x_3|>\epsilon} \textit{ if } s<\frac{k}{n}\\
        &&\frac{1}{2\epsilon}\mathbb{I}_{x_3\in [-\epsilon,\epsilon]}|x_2|\mathbb{I}_{sgn(x_1)\neq sgn(x_2)}\mathbb{I}_{|x_2|>\epsilon}\textit{ otherwise.}\\
      \end{array}
    \right.
    \end{equation*}
Now $F^\epsilon$ satisfies the hypothesis $(H_2)$ and $(H_3)$ of Proposition \ref{mainBound}, with $p=1, q=2$, $h=\frac{1}{n\min\left(s,\frac{k}{n},s-\frac{k}{n}\right)}$,
 $t_1=s, t_2=\frac{k}{n}, t_3=\frac{k+1}{n}$ if $s<\frac{k}{n}$ and $t_1=\frac{k}{n}, t_2=\frac{k+1}{n}, t_3=s$ otherwise. 

\medskip 

Moreover, $(H_1)$ is satisfied with $M=3$ and
\begin{equation*}G^\epsilon(x):=\left\{
      \begin{array}{cll}
        &&\frac{1}{2\epsilon}|x_3|\mathbb{I}_{x_1\in [-\epsilon,\epsilon]} \textit{ if } s<\frac{k}{n}\\
        &&\frac{1}{2\epsilon}|x_2|\mathbb{I}_{x_3\in [-\epsilon,\epsilon]}\textit{ otherwise.}\\
      \end{array}
    \right.
\end{equation*}

For simplicity, assume that $s>\frac{k}{n}$ (the case $s<\frac{k}{n}$ can be handled in the exact same way). Noticing that the density of a vector $(X_1,X_2)$ evaluated in $(a,a)$ is the same as
 the density of $(X_1,X_2-X_1)$ evaluated in $(a,0)$, we have, using \eqref{complete} in Proposition \ref{mainBound}, for all $\epsilon>0$, that

\begin{eqnarray*}
&&\left|A^{\epsilon}_{s,\frac{k}{n}}-\phi_{\frac{k}{n},s}(a,a)\int_{\mathbb{R}^2}\frac{1}{2\epsilon}\mathbb{E}[F^{\epsilon}(\tilde B(x))]dx_1dx_2\right|\\
&\leq&\frac{\beta h^{2-2H}e^{-Da^2}}{\prod_{i=1}^3|t_{i}-t_{i-1}|^H}\int_{\mathbb{R}^2}\left(\int_{\mathbb{R}}G^{\epsilon}(y_1,y_2,y_3)dy_3\right)\\
&&\times|y_1+y_2|\mathbb{I}_{sgn(y_1)\neq sgn(y_1+y_2)}e^{-D\left(\frac{k+1}{n}\wedge t-\frac{k}{n}\right)^{2H}y_2^2}dy_1dy_2\\
&\leq&\beta h^{2-2H}e^{-Da^2}\frac{1}{\left(\frac{k}{n}\right)^H\left|s-\frac{k}{n}\right|^Hn^{2H}},
\end{eqnarray*}
where $\tilde B(x)$ is defined in the subsection \ref{notTech}, $\phi_{u,v}$ is defined in \eqref{density}, and where we used Lemma \ref{calculus1} in the last line.

We also have
\begin{eqnarray*}
&&\int_{\mathbb{R}^2}\frac{1}{2\epsilon}\mathbb{E}[F^{\epsilon}(\tilde B(x))]dx_1dx_2\\
&=&\int_{\mathbb{R}^2}\left(\frac{1}{2\epsilon}\int_{\mathbb{R}}\mathbb{I}_{x_3\in [x_1+x_2-\epsilon,x_1+x_2+\epsilon]}dx_3\right)\mathbb{I}_{|x_2|>\epsilon}\\
&&\times\frac{|x_1+x_2|\mathbb{I}_{sgn(x_1)\neq sgn(x_1+x_2)}e^{-\left(\frac{k+1}{n}\wedge t-\frac{k}{n}\right)\frac{x_1^2}{2}}}{\sqrt{2\pi\left(\frac{k+1}{n}\wedge t-\frac{k}{n}\right)^H}}dx_1dx_2\\
&\underset{\epsilon\rightarrow 0}\longrightarrow&\int_{\mathbb{R}^2}\frac{|x_1+x_2|\mathbb{I}_{sgn(x_1)\neq sgn(x_1+x_2)}e^{-\left(\frac{k+1}{n}\wedge t-\frac{k}{n}\right)\frac{x_1^2}{2}}}{\sqrt{2\pi\left(\frac{k+1}{n}\wedge t-\frac{k}{n}\right)^H}}dx_1dx_2\\
&=&\frac{1}{2}\left(\frac{k+1}{n}\wedge t-\frac{k}{n}\right)^{2H},
\end{eqnarray*}
where we used Lebesgue's dominated convergence theorem and where the last equality follows from Lemma \ref{calculus1}. In the similar way, the case $s<\frac{k}{n}$ yields the exact same bound and thus, we obtain
\begin{eqnarray}\label{termePrincipal}
&&\lim\sup_{\epsilon\rightarrow 0}\left|A^{\epsilon}_{s,\frac{k}{n}}-\frac{\left(\left(\frac{k}{n}+\frac{1}{n}\right)\wedge t-\frac{k}{n}\right)^{2H}\phi_{\frac{k}{n},s}(a,a)}{2}\right|\notag\\
&\leq&\beta h^{2-2H}e^{-Da^2}\frac{1}{\left(\frac{k}{n}\right)^H\left|s-\frac{k}{n}\right|^Hn^{2H}}.
\end{eqnarray}

\item Let next $\min\left(s,\frac{k}{n},\left|s-\frac{k}{n}\right|\right)<\frac{2}{n}$. In this case we can use \eqref{Basic} in Proposition \ref{mainBound} to obtain that
\begin{eqnarray*}
A^{\epsilon}_{s,\frac{k}{n}}&\leq&\frac{\beta e^{-Da^2}}{\min\left(s,\frac{k}{n}\right)^H\left|s-\frac{k}{n}\right|^H\frac{1}{\left(\frac{k+1}{n}\wedge t-\frac{k}{n}\right)^H}}\\
&&\times\int_{\mathbb{R}^2}|x+y|\mathbb{I}_{sgn(x)\neq sgn(y)}e^{-D\frac{\left(\frac{k+1}{n}\wedge t-\frac{k}{n}\right)^{2H}x^2}{2}}dxdy\\
&\leq&\beta\frac{n^{-2H}}{min\left(s,\frac{k}{n}\right)^H\left|s-\frac{k}{n}\right|^H}e^{-Da^2},
\end{eqnarray*}
where the last line follows again from Lemma \ref{calculus1}.
\end{itemize}

\bigskip

Now we are ready to estimate terms $R^{1,n}-R^{4,n}$. For $R^{2,n}$, the above bound gives 
\begin{eqnarray*}
|R^{2,n}|&\leq& \beta n^{2H}n^{-2H}\int_0^\frac{2}{n}\int_0^te^{-Da^2}\frac{1}{u^H\left(|u-v|^H+|u-v_n|^H\right)}dudv\\
&\leq&\beta n^{2H}n^{-2H}\int_0^\frac{2}{n}\int_0^te^{-Da^2}\frac{1}{u^H\left(3|u-v|^H+\left|\frac{1}{n}-u\right|^H\mathbb{I}_{v_n=\frac{1}{n}}\right)}dudv\\
&\leq& \frac{\beta}{n}e^{-Da^2}\int_0^\frac{2}{n}\frac{1}{u^H\left(\frac{2}{n}-u\right)^H}+K\int_0^\frac{2}{n}\int_\frac{2}{n}^t\frac{1}{3|v-\frac{1}{n}|^Hu^H}dudv\\
&\leq& \beta n^{1-H}e^{-Da^2},
\end{eqnarray*}
and the terms $R^{3,n}, R^{4,n}$ can be treated similarly. Finally, thanks to \eqref{termePrincipal}, we have

\begin{eqnarray*}
|R^{1,n}|&\leq& \int_{[0,T]^2}dudv\left(\frac{\beta e^{-Da^2} h^{2-2H}}{min\left(v_n,u\right)^H\left|v_n-u\right|^H}\right.\\
&&\left.+\left|\frac{\phi_{u,v}(a,a)}{2}-n^{2H}\frac{\left(\left(u_{n}+\frac{1}{n}\right)\wedge t-u_n\right)^{2H}}{2}\phi_{u,v_n}(a,a)\right|\right)\mathbb{I}_{(u,v)\in\mathcal C},
\end{eqnarray*}
with $\mathcal{C}=\left\{(u,v)\in [0,T]^2, \min(u,v_n,|u-v_n|)\geq\frac{2}{n}\right\}$.

According to Lemma \ref{calculus3}, we then have
\begin{eqnarray*}
|R^{1,n}|&\leq&\int_0^t\int_{t-\frac{1}{n}}^{t}\phi_{u,v}(a,a)dvdu\\
&&+\int_{[0,T]^2}\frac{\beta e^{-Da^2} h^{2-2H}}{min\left(v_n,u\right)^H\left|v_n-u\right|^H}\mathbb{I}_{(u,v)\in\mathcal C}dudv+ \beta e^{-Da^2}n^{H-1}\\
&\leq& \beta e^{-\frac{a^2}{2T}}n^{H-1}\\
&&+e^{-Da^2}\left(\int_{[0,T]^2}\frac{\beta e^{-Da^2} h^{2-2H}}{4^Hmin\left(v,u\right)^H\left|v-u\right|^H}\mathbb{I}_{(u,v)\in\mathcal C}dudv+n^{H-1}\right)\\
&\leq&e^{-Da^2}\left(\int_{[0,T]^2}\beta e^{-Da^2} n^{2-2H}\left(\frac{1}{4^Hmin\left(v,u\right)^H\left|v-u\right|^{1+H}}\right.\right.\\
&&\left.\left.+\frac{1}{4^Hmin\left(v,u\right)^{1+H}\left|v-u\right|^{H}}\right)\mathbb{I}_{(u,v)\in\mathcal C}dudv+n^{H-1}\right)\\
&\leq& \beta e^{-Da^2}n^{H-1}.
\end{eqnarray*}
The bound \eqref{mixed} is then proved.
\end{proof}

\begin{proof}[Proof\nopunct]\textit{of \eqref{squared}:}
\bigskip In order to prove the bound \eqref{squared}, we can see that 
$$\mathbb{E}[S_n(t)^2]=\sum_{k,j=0}^{\lfloor nt\rfloor}A_{\frac{k}{n},\frac{j}{n}},$$
with 

\begin{eqnarray*}
A_{\frac{k}{n},\frac{j}{n}}&:=&\mathbb{E}\left[F\left(B_\frac{k}{n},B_{\frac{k+1}{n}\wedge t},B_\frac{j}{n},B_{\frac{j+1}{n}\wedge t}\right)\right],\\
F:(u,v,w,x)&\rightarrow&|v-a|\mathbb{I}_{sgn(u-a)\neq sgn(v-a)}|x-a|\mathbb{I}_{sgn(w-a)\neq sgn(x-a)}.
\end{eqnarray*}

The function $F$ satisfies the hypotheses of Proposition \ref{mainBound} with
 $q=2$, $p=2$ and $h=\frac{1}{n\min\left(\frac{k}{n},\frac{j}{n}, |\frac{j}{n}-\frac{k+1}{n}|,|\frac{j+1}{n}-\frac{k}{n}|\right)}$.
 Finally, applying the Proposition \ref{mainBound} and proceeding as in the proof of \eqref{mixed}, we obtain that when $\min\left(\frac{k}{n},\frac{j}{n},\frac{|k-j|}{n}\right)>\frac{2}{n}$,
 
\begin{eqnarray*}
&&\left|A_{\frac{k}{n},\frac{j}{n}}-\frac{\left(\frac{k+1}{n}\wedge t-\frac{k}{n}\right)^{2H}\left(\frac{j+1}{n}\wedge t-\frac{j}{n}\right)^{2H}}{4}\phi_{\frac{k}{n},\frac{j}{n}}(a,a)\right|\\
&\leq&\beta e^{-Da^2}h^{2-2H}\frac{n^{-4H}}{\min\left(\left(\frac{k}{n}\right)^H,\left(\frac{j}{n}\right)^H\right)\left|\frac{k}{n}-\frac{j}{n}\right|^H}.
\end{eqnarray*}

Thus, proceeding exactly as in the proof of \eqref{mixed} and applying Lemma \ref{calculus3}, we obtain the desired bound.
\end{proof}

\subsection{Step 3: general case}
We have established in Step 1 and Step 2 that for some constant $D,\beta$, when $f=\mathbb{I}_{\cdot>a}$, and writing 

$$S_{n,a}^{i,j}(t):=n^{2H-1}\left(\int_0^t\mathbb{I}_{B^i_s>a}dB^j_s-\sum_{k=0}^{\lfloor nt\rfloor}\mathbb{I}_{B^i_{\frac{k}{n}}>a}\left(B^j_{\frac{k+1}{n}\wedge t}-B^j_\frac{k}{n}\right)\right),$$ 
we have for all $t\in [0,T]$,

$$\left\|S_{n,a}^{i,j}(t)-\frac{\delta_{ij}}{2}L^i_t(a)\right\|_{L^2(\Omega)}\leq \beta n^{-\frac{1-H}{2}}e^{-Da^2}.$$

Moreover, it is well known (see for example p. 224 in \cite{revuz2013continuous}) that we can write for all $x\in I$ with $I\subset \mathbb{R}$ an interval,
$$f(x)=\beta_I+\int_I sgn(x-a)\mu(da).$$
It follows that for $x,y\in I$ we have
\begin{eqnarray*}f(x)-f(y)&=&\int_I(sgn(x-a)-sgn(y-a))\mu(da)\\
&=&2\int_{\mathbb{R}}(\mathbb{I}_{x>a}-\mathbb{I}_{y>a})\mu(da)
\end{eqnarray*}
and hence the formula follows for arbitrary $x,y\in \mathbb{R}$.

Assume now that $\mu$ verifies the growth condition \eqref{growthC} with $P=D$, where $D$ is as in Step 1 and Step 2.
Using Jensen's inequality and Fubini's theorem (which is possible since $\mu$ is a $\sigma$-finite Radon measure), we get

\begin{eqnarray*}
&&\left\|S_n^{i,j}(t)-\delta_{ij}\int_{\mathbb{R}}L^i_t(a)d\mu(a)\right\|^2_{L^2(\Omega)}\\
&\leq&4\left\|\int_{\mathbb{R}}\left|S_{n,a}^{i,j}(t)-\delta_{ij}L^i_t(a)\right|d\mu(a)\right\|^2_{L^2(\Omega)}\\
&\leq&4\left\|\int_{\mathbb{R}}\frac{\left|S_{n,a}^{i,j}(t)-\delta_{ij}L^i_t(a)\right|}{\exp(-\frac{Pa^2}{2})}e^{-\frac{Pa^2}{2}}d\mu(a)\right\|^2_{L^2(\Omega)}\\
&\leq&4\int_{\mathbb{R}}e^{-\frac{Pa^2}{2}}d\mu(a)\int_{\mathbb{R}}\frac{\|S^{n,a}_{i,j}(t)-\frac{\delta_{ij}}{2}L_t(a)\|^2_{L^2(\Omega)}}{\exp(-Pa^2)}d\mu(a)\\
&\leq&K\beta n^{1-H}\int_{\mathbb{R}}e^{-\frac{Pa^2}{2}}d\mu(a)\int_{\mathbb{R}}e^{-Pa^2}d\mu(a)\\
&\leq& C(H,T)\left(\int_{\mathbb{R}}e^{-\frac{Pa^2}{2}}d\mu(a)\right)^2n^{1-H}.
\end{eqnarray*}
which concludes the proof of Theorem \ref{main}.
\begin{remark}
\label{remark:BV-proof}
The claim of Remark \ref{remark:BV} follows directly from the argument of Step 3, if we use Jordan decomposition and write a function of bounded variation as $f=f_1-f_2$, where $f_1$ and $f_2$ are derivatives of convex functions. Then $\mu = f' = \mu_1 - \mu_2$, and it simply remains to use total variation measure $|\mu| = \mu_1 + \mu_2$ in the upper bounds. 
\end{remark}

\section{Proof of Proposition \ref{mainBound}}\label{prfMain2}
Let us introduce the following notation:
\begin{eqnarray}\label{D}\mathcal D&=&\llbracket 1,p+q\rrbracket^2\setminus J\times J,\end{eqnarray} 
where $J$ is as in the statement of Proposition \ref{mainBound}.

\smallbreak 

We will start by proving the following lemma, which provides a useful decomposition of the determinant of the vector $(B_{t_k}-B_{t_{k-1}})_{k\in\llbracket 1,p+q\rrbracket}$.

\begin{lemma}\label{decompDet}
Let $\Sigma:=\Sigma\left(\left(0,t_1\right),\ldots,(t_{p+q},t_{p+q-1}))\right)$ and\\ $\Sigma'=\Sigma\left(\left(t_{\bar J[i]},t_{\bar J[i-1]})_{i\in\llbracket 1,q\rrbracket}\right)\right)$.
 Assume that the hypothesis $(H_2)$ and $(H_3)$ of Proposition \ref{mainBound} are satisfied. Then, there is a constant $\beta>0$ (independent from $(t_1,\ldots,t_{p+q})$)
 and functions $\Theta_j:\mathbb{R}\rightarrow\mathbb{R}, j\in\{1,2,3,4\}$ such that for all  $ 0<h<1$ we have $|\Theta_j(h)|\leq \beta h^{2-2H}$ and

\begin{enumerate}
\item[(a.)] \begin{eqnarray}\label{interm1}
det\left(\Sigma\right)=det\left(\Sigma'\right)\prod_{i=1}^p|t_{J_{[i]}}-t_{J_{[i]}-1}|^{2H}\left(1+\Theta_1(h)\right).
\end{eqnarray}

\medskip

\item[(b.)] For $i\in J$,
\begin{equation}\label{Cofacteur1}\Sigma^{-1}_{ii}=\frac{1}{|t_i-t_{i-1}|^{2H}}\left(1+\Theta_2(h)\right)\end{equation}

\medskip

\item[(c.)] For $i\neq j\in J$,
\begin{equation}\label{Cofacteur2}\Sigma^{-1}_{ij}=\frac{1}{|t_i-t_{i-1}|^{H}|t_j-t_{j-1}|^H}\Theta_3(h)\end{equation}

\medskip

\item[(d.)] For $i,j\in \llbracket 1,q\rrbracket$,
\begin{equation}\label{Cofacteur3}\Sigma^{-1}_{\bar J[i],\bar J[j]}=\Sigma'^{-1}_{ij}\left(1+\Theta_4(h)\right)\end{equation}

\item[(e.)] For $(i,j)\in\mathcal D$,
\begin{equation}\label{Cofacteur4}
|\Sigma^{-1}_{ij}|\leq \frac{\beta}{\max\left(|t_i-t_{i-1}||t_j-t_{j-1}|^{2H-1},|t_i-t_{i-1}|^{2H-1}|t_j-t_{j-1}|\right)}.
\end{equation}
\end{enumerate}

\end{lemma}

We will only do the proof of items (a.) and (b.) since the proof of the other items is similar after writing $\Sigma^{-1}_{ij}=\frac{(-1)^{i+j}C_{ij}}{det(\Sigma)}$, where $C$ is the cofactor matrix of $\Sigma$. Throughout the proof, we denote by $\Theta(h)$ a generic function that satisfies $|\Theta(h)|\leq \beta h^{2-2H}$.

\medskip

\begin{proof}[Proof\nopunct]\textit{of item (a.):}

The proof is divided into two steps. In the first step, we show an intermediate result 
\begin{eqnarray}\label{interm2}&&det\left(\Sigma\left(\left(t_{\bar J[i]},t_{\bar J[i]-1}\right)_{i\in\llbracket 1,q\rrbracket}\right)\right)=det(\Sigma')(1+\Theta(h)),\notag\\
\end{eqnarray} 
where the notation $\Sigma((t^i_1,t^i_{2}))_{i\leq m}$ is defined in \eqref{covariance}, before finishing the actual proof in the second step. 

\medskip

\textit{\underline{Step 1: Proof of \eqref{interm2}}}
Thanks to hypothesis $(H_1)$, we have that for all $i\in\llbracket 1,q\rrbracket$,

 $$\left|(t_{\bar J[i]}-t_{\bar J[i-1]})-(t_{\bar J[i]}-t_{\bar J[i]-1})\right|\leq h\min_{l\in \llbracket 1,q\rrbracket}\left|t_{\bar J[l]}-t_{\bar J[l]-1}\right|.$$

Furthermore, thanks to Lemma \ref{estimates} and the H\"older inequality, we have for all $i,j\in\llbracket 1,q\rrbracket$,
\begin{eqnarray*}
&&\left|\mathbb{E}\left[\left(B_{t_{J[j]}}-B_{t_{J[j]-1}}\right)\left(B_{t_{J[i]}}-B_{t_{J[i]-1}}\right)\right]\right.\\
&&\left.-\mathbb{E}\left[\left(B_{t_{J[j]}}-B_{t_{J[j-1]}}\right)\left(B_{t_{J[i]}}-B_{t_{J[i-1]}}\right)\right]\right|\\
&\leq&h\min_{j\in \llbracket 1,q\rrbracket}\left|t_{\bar J[l]}-t_{\bar J[l]-1}\right|\left(|t_{\bar J[i]}-t_{\bar J[i]-1}|^{2H-1}+|t_{\bar J[j]}-t_{\bar J[j]-1}|^{2H-1}\right)\\
&&+ h^{2H}\min_{j\in \llbracket 1,q\rrbracket}\left|t_{\bar J[l]}-t_{\bar J[l]-1}\right|^{2H}\\
&\leq& 3h|t_{\bar J[j]}-t_{\bar J[j]-1}|^H|t_{\bar J[i]}-t_{\bar J[i]-1}|^H.
\end{eqnarray*}

We then have 
\begin{eqnarray*}
&&det\left(\Sigma\left(\left(t_{\bar J[i]},t_{\bar J[i]-1}\right)_{i\in\llbracket 1,q\rrbracket}\right)\right)\\
&=&\sum_{\tau\in \mathfrak S_{\llbracket 1,q\rrbracket}}sgn(\tau)\prod_{i=1}^{q}\mathbb{E}\left[\left(B_{t_{\bar J[j]}}-B_{t_{\bar J[j]-1}}\right)\left(B_{t_{\bar J[\tau(j)]}}-B_{t_{\bar J[\tau(j)]-1}}\right)\right]\\
&=&det\left(\Sigma\left(\left(t_{\bar J[i]},t_{\bar J[i-1]}\right)_{i\in\llbracket 1,q\rrbracket}\right)\right)+G(h),
\end{eqnarray*}
where $\mathfrak S_{\llbracket 1,q\rrbracket}$ is the permutation group of $\llbracket 1,q\rrbracket$ and
\begin{eqnarray*}|G(h)|&\leq& \sum_{r=1}^{q}\binom{q}{r}(3h)^r\sum_{\tau\in\mathfrak S_{\llbracket 1,q\rrbracket}}\prod_{i=1}^{q}|t_{\bar J[i]}-t_{\bar J[i]-1}|^H|t_{\bar J[\tau(i)]}-t_{\bar J[\tau(i)]-1}|^H\\
&\leq& h3^{q}2^{q}q!\prod_{i=1}^{q}|t_{\bar J[i]}-t_{\bar J[i]-1}|^{2H}\\
&\leq& h^{2-2H}\frac{3^{q}2^{q}(q)!}{k_H}det\left(\Sigma'\right),
\end{eqnarray*}
where the last inequality is obtained with \eqref{Sandwich}. This conclude the proof of the first step.

\medskip

\textit{\underline{Step 2: Proof of \eqref{interm1}}}
We factorise the determinant as

\begin{eqnarray*}
det\left(\Sigma\right)
&=&det\left(\Sigma\left(\left(t_{\bar J[i]-1},t_{\bar J[i]}\right)_{i\in\llbracket 1,q\rrbracket}\right)\right)det\left(\Sigma\left(\left(t_{ J[i]-1},t_{ J[i]}\right)_{i\in\llbracket 1,p\rrbracket}\right)\right)\\
&&+\sum_{\tau\in\mathfrak H}sgn(\tau)\prod_{j=1}^{p+q}\mathbb{E}\left[\left(B_{t_j}-B_{t_{j-1}}\right)\left(B_{t_{\tau(j)}}-B_{t_{\tau(j)-1}}\right)\right]\\
\end{eqnarray*}
where $\mathfrak H$ is the set of permutations of $\llbracket 1,p+q\rrbracket$ that does not stabilize $\bar J$, i.e. there exists at least one element in $\bar J$ whose image under the permutation belongs to $J$.

\begin{itemize}
\item Similarly as in step 1, we have that 
\begin{eqnarray*}
&&det\left(\Sigma\left(\left(t_{J[i]},t_{ J[i]-1}\right)_{i\in\llbracket 1,p\rrbracket}\right)\right)=\prod_{i=1}^p|t_{J[i]}-t_{J[i]-1}|^{2H}\\
&&+\sum_{\kappa\in\mathfrak  S_{\llbracket 1,p\rrbracket},\kappa\neq Id}sgn(\kappa)\prod_{j=1}^p\mathbb{E}\left[\left(B_{t_{J[j]}}-B_{t_{J[j]-1}}\right)\left(B_{t_{J[\kappa(j)]}}-B_{t_{J[\kappa(j)]-1}}\right)\right].
\end{eqnarray*}
Let $R(j,\tau):=\mathbb{E}\left[\left(B_{t_{J[j]}}-B_{t_{J[j]-1}}\right)\left(B_{t_{J[\kappa(j)]}}-B_{t_{J[\kappa(j)]-1}}\right)\right]$.
For any $\kappa\in\mathfrak S_{\llbracket 1,p\rrbracket}$, we have that
\begin{equation*}\left\{
      \begin{array}{cll}
        R(j,\tau)&=& |t_{J[j]}-t_{J[j]-1}|^{2H} \textit{ if } \kappa(j)=j\\
        |R(j,\tau)|&\leq& \frac{|t_{J[j]}-t_{J[j]-1}||t_{J[\kappa(j)]}-t_{J[\kappa(j)]-1}|}{|t_{max\{j-1,\kappa(j)-1\}}-t_{min\{j,\kappa(j)\}}|^{2-2H}} \textit{ if } \kappa(j)\neq j,\\
      \end{array}
    \right.
    \end{equation*}
where the inequality in the second case comes from Lemma \ref{estimates}.
 Thanks to the hypothesis $(H_2)$, we have that if $\tau(j)\neq j$, there is at least one element $k\in \bar J$ such that $k\in \llbracket t_{min\{j,\tau(j)\}+1},t_{max\{j-1,\tau(j)-1\}}\rrbracket$. Hence, 
\begin{eqnarray*}
&&\frac{|t_{J[j]}-t_{J[j]-1}||t_{J[\kappa(j)]}-t_{J[\kappa(j)]-1}|}{|t_{max\{j-1,\kappa(j)-1\}}-t_{min\{j,\kappa(j)\}}|^{2-2H}}\\
&\leq& h^{2-2H}|t_{J[j]}-t_{J[j]-1}|^H|t_{J[\kappa(j)]}-t_{J[\kappa(j)]-1}|^H.
\end{eqnarray*}
Finally, for a permutation $\kappa\in\mathfrak S_{\llbracket 1,p\rrbracket}$, let $T_{\kappa}\subset\llbracket 1,p\rrbracket$ denotes the set of points $j$ such that $\kappa(j)\neq j$. We have 

\begin{eqnarray*}
&&\left|\sum_{\kappa\in\mathfrak  S_{\llbracket 1,p\rrbracket},\kappa\neq Id}sgn(\kappa)\prod_{j=1}^p\mathbb{E}\left[\left(B_{t_{J[j]}}-B_{t_{J[j]-1}}\right)\left(B_{t_{J[\kappa(j)]}}-B_{t_{J[\kappa(j)]-1}}\right)\right]\right|\\
&\leq&\sum_{\kappa\in\mathfrak  S_{\llbracket 1,p\rrbracket},\kappa\neq Id}h^{Card(T_\kappa)(2-2H)}\prod_{j=1}^p|t_{J[j]}-t_{J[j]-1}|^H|t_{J[\tau(j)]}-t_{J[\tau(j)]-1}|^H\\
&\leq& h^{2-2H}p!\prod_{j=1}^p|t_{J[j]}-t_{J[j]-1}|^{2H}\\
&\leq& \Theta(h)\prod_{j=1}^p|t_{J[j]}-t_{J[j]-1}|^{2H}.
\end{eqnarray*}

\item Thanks to the first step, we also have that 
\begin{eqnarray*}
det\left(\Sigma\left(\left(t_{\bar J^1_{\sigma}[i]},t_{\bar J^1_{\sigma}[i]-1}\right)_{i\in\llbracket 1,p+q\rrbracket}\right)\right)=det\left(\Sigma'\right)(1+\Theta(h)).
\end{eqnarray*}
This combined with the previous point leads to
\begin{eqnarray*}
&&det\left(\Sigma\left(\left(t_{\bar J[i]},t_{\bar J[i]-1}\right)_{i\in\llbracket 1,q\rrbracket}\right)\right)det\left(\Sigma\left(\left(t_{ J[i]},t_{ J[i]-1}\right)_{i\in\llbracket 1,p\rrbracket}\right)\right)\\
&=&det\left(\Sigma'\right)\prod_{i=1}^p|t_{J[i]}-t_{J[i]-1}|^{2H}
\left(1+\Theta(h)\right).
\end{eqnarray*}

\item Finally, let us analyze the expression
$$\sum_{\tau\in\mathfrak H}sgn(\tau)\prod_{j=1}^{p+q}\mathbb{E}\left[\left(B_{t_j}-B_{t_{j-1}}\right)\left(B_{t_{\tau(j)}}-B_{t_{\tau(j)-1}}\right)\right].$$

Since $\tau$ does not stabilize $\bar J$, there is a couple $(k_0,k_1)\in \bar J\times J$ such that $\tau(k_0)\in J$ and $\tau(k_1)\in \bar J$. Thanks again to Lemma \ref{estimates}, we have 
\begin{eqnarray*}
&&\mathbb{E}\left[\left(B_{t_{k_0}}-B_{t_{k_0-1}}\right)\left(B_{t_{\tau(k_0)}}-B_{t_{\tau(k_0)-1}}\right)\right]\\
&\leq& |t_{\tau(k_0)}-t_{\tau(k_0)-1}||t_{k_0}-t_{k_0-1}|^{2H-1}\\
&\leq& h^{1-H}|t_{k_0}-t_{k_0-1}|^H|t_{\tau(k_0)}-t_{\tau(k_0)-1}|^H.
\end{eqnarray*}

and
\begin{eqnarray*}
&&\mathbb{E}\left[\left(B_{t_{k_1}}-B_{t_{k_1-1}}\right)\left(B_{t_{\tau(k_1)}}-B_{t_{\tau(k_1)-1}}\right)\right]\\
&\leq& h^{1-H}|t_{k_1]}-t_{k_1-1}|^H|t_{\tau(k_1)}-t_{\tau(k_1)-1}|^H.
\end{eqnarray*}

Also, for $k\neq k_0,k_1$, we can use the H\"older inequality to obtain
\begin{eqnarray*}&&\mathbb{E}\left[\left(B_{t_{k}}-B_{t_{k-1}}\right)\left(B_{t_{\tau(k)]}}-B_{t_{\tau(k)-1}}\right)\right]\\
&\leq& |t_{k}-t_{k-1}|^H|t_{\tau(k)}-t_{\tau(k)-1}|^H. 
\end{eqnarray*}

Putting all these elements together, we have
\begin{eqnarray*}
&&\left|\sum_{\tau\in\mathfrak H}sgn(\tau)\prod_{j=1}^{p+q}\mathbb{E}\left[\left(B_{t_j}-B_{t_{j-1}}\right)\left(B_{t_{\tau(j)}}-B_{t_{\tau(j)-1}}\right)\right]\right|\\
&\leq&(p+q)!h^{2-2H}\prod_{i=1}^{p+q}|t_{k}-t_{k-1}|^H|t_{\tau(k)}-t_{\tau(k)-1}|^H\\
&\leq&(p+q)!h^{2-2H}\prod_{i=1}^{p+q}|t_{k}-t_{k-1}|^{2H}\\
&\leq&(p+q)!h^{2-2H}\prod_{i=1}^{q}|t_{\bar J[k]}-t_{\bar J[k-1]}|^{2H}\prod_{i=1}^p|t_{J[k]}-t_{ J[k-1]}|^{2H}\\
&\leq&(p+q)!\frac{h^{2-2H}}{k_H}det\left(\Sigma\left(\left(t_{\bar J[i]},t_{\bar J[i-1]}\right)_{i\in\llbracket 1,q\rrbracket}\right)\right)\prod_{i=1}^p|t_{J[i]}-t_{J[i]-1}|^{2H}.
\end{eqnarray*}
where the last inequality follows from the local non-deteminism property, see \eqref{Sandwich}. Thus, the proof is complete.
\end{itemize}
\end{proof}

\begin{proof}[Proof\nopunct]\textit{of \eqref{Cofacteur1}:}
Let $i\in J$. We have that $\Sigma^{-1}_{ii}=\frac{C_{ii}}{det(\Sigma)}$, with $C$ the cofactor matrix of $\Sigma$.
\begin{itemize}
\item We have established that 
$$det(\Sigma)=det(\Sigma')\prod_{k=1}^p|t_{J[i]}-t_{J[k-1]}|^{2H}(1+\Theta(h)).$$
\item Proceeding in the exact same way, we can show that 
$$C_{ii}=det(\Sigma')\prod_{k\neq i}|t_{J[k]}-t_{J[k-1]}|^{2H}(1+\Theta(h)).$$
\end{itemize} 
This complete the proof of \eqref{Cofacteur1}.
\end{proof}

We will also need the following elementary result. 
\begin{lemma}\label{Diagonalisation}

Let $0\leq s_0<s_1\leq s_2<s_3\ldots\leq s_{2m-2}< s_{2m-1}\leq T$ and $\Sigma:=\Sigma\left((s_{2i-1},s_{2i-2})_{i\in\llbracket 1,m\rrbracket}\right)$ be as in \ref{covariance}.
 Then, there is a constant $D=D(m,H,T)>0$ independent of the choice of $s_i$ such that for all $x\in\mathbb{R}^m$,
$$x^T\Sigma^{-1}x\geq D(m,H,T)\sum_{i=1}^m\frac{x_i^2}{|s_{2i-1}-s_{2i-2}|^{2H}}.$$

\end{lemma}

\begin{proof}
Notice that the result is trivially valid when $m=1$. We proceed by induction on $m$. Assume now that the result is true for every $k\leq m$ for some $m\in\mathbb{N}^*$,
 and let us consider $(s_i)_{i\in\llbracket 0, 2m+1\rrbracket}$ verifying the same condition as in the statement.
 Without loss of generality, we suppose that increments are ordered such that $\forall i\in\llbracket 1,m\rrbracket,~|s_{2i-1}-s_{2i-2}|\geq |s_{2i+1}-s_{2i}|$.

\smallbreak

Notice first that if $l\in\llbracket1,m\rrbracket$, a slight modification of \eqref{Cofacteur3} and \eqref{Cofacteur4} implies the existence of a function $\Theta$, that is independent from the choice of $s_i$, 
 such that $|\Theta(h)|\leq \beta h^{2-2H}$ and for all $i,j>l$,
\begin{equation}\label{CofacteurBis1}
\Sigma^{-1}\left((s_{2k-1},s_{2k-2})_{k\in\llbracket 1,m+1\rrbracket}\right)_{ij}=\Sigma^{-1}\left((s_{2k-1},s_{2k-2})_{k\in\llbracket l+1,m+1\rrbracket}\right)_{ij}(1+\Theta(h_l)),
\end{equation}
with $h_l=\frac{|s_{2l+1}-s_{2l}|}{|s_{2l-1}-s_{2l-2}|}$, and 
\begin{equation}\label{CofacteurBis2}
\left|\Sigma^{-1}\left((s_{2k-1},s_{2k-2})_{k\in\llbracket 1,m+1\rrbracket}\right)_{ij}\right|\leq\frac{K}{|s_{2i-1}-s_{2i-2}|^H|s_{2j-1}-s_{2j-2}|^{H}}.
\end{equation}

\smallbreak

 We introduce a cut-off parameter $\kappa>0$ such that
 \begin{equation}
 \label{kappa}
 \max(\sqrt{\kappa},|\Theta(\kappa)|)\leq\left(\frac{D(m)}{16K(m+1)^2}\right)
 \end{equation}
and let $x\in\mathbb{R}^{m+1}$. We will distinguish several cases:

\begin{enumerate}
 \item[(i.)] Assume that for all $i\in\llbracket 1,m\rrbracket$, $\frac{|s_{2i+1}-s_{2i}|}{|s_{2i-1}-s_{2i-2}|}\geq\kappa$. Then we have,
 according to \eqref{eigenvalues}, $$x^T\Sigma^{-1}x\geq \frac{1}{(m+1)|s_1-s_0|^{2H}}\sum_{i=1}^{m+1}x_i^2\geq\frac{\kappa^{2H(m+1)}}{(m+1)}\sum_{i=1}^{m+1}\frac{x_i^2}{|s_{2i-1}-s_{2i-2}|^{2H}}.$$
\item[(ii.)] Assume now that the previous condition does not hold true and let $i_0=\min\left\{i\in\llbracket 1,m\rrbracket |\frac{|s_{2i+1}-s_{2i}|}{|s_{2i-1}-s_{2i-2}|}<\kappa\right\}.$ 
\end{enumerate}

Let us again distinguish two sub-cases:

\begin{enumerate}
\item[(ii.1)] Assume that 
\begin{equation}\label{conditionMax}\max_{i\in\llbracket 1,i_0\rrbracket}\frac{x_i^2}{|s_{2i-1}-s_{2i-2}|^{2H}}\geq\kappa\max_{i\in\llbracket i_0+1,m+1\rrbracket}\frac{x_i^2}{|s_{2i-1}-s_{2i-2}|^{2H}}.\end{equation}
In this case, we have again thanks to \eqref{eigenvalues},
\begin{eqnarray*}
x^T\Sigma^{-1}x&\geq& \frac{1}{m+1}\frac{1}{|s_1-s_0|^{2H}}\sum_{i=1}^{m+1}x_i^2\\
&\geq& \frac{1}{m+1}\max_{i\in\llbracket 1,i_0\rrbracket}\frac{x_i^2}{|s_{2i-1}-s_{2i-2}|^{2H}}\frac{|s_{2i-1}-s_{2i-2}|^{2H}}{|s_1-s_0|^{2H}}\\
&\geq&\frac{\kappa^{2H(i_0-1)}}{(m+1)}\max_{i\in\llbracket 1,i_0\rrbracket}\frac{x_i^2}{|s_{2i-1}-s_{2i-2}|^{2H}}\\
&\geq&\frac{\kappa^{2H(i_0-1)+1}}{(m+1)}\max_{i\in\llbracket 1,m+1\rrbracket}\frac{x_i^2}{|s_{2i-1}-s_{2i-2}|^{2H}}\\
&\geq&\frac{\kappa^{2H(m+1)}}{(m+1)^2}\sum_{i=1}^{m+1}\frac{x_i^2}{|s_{2i-1}-s_{2i-2}|^{2H}},
\end{eqnarray*}
where the second inequality comes from the fact that for all $i\leq i_0$, $\frac{|s_{2i-1}-s_{2i-2}|^{2H}}{|s_1-s_0|^{2H}}\geq\kappa^{2H(i-1)}$.

\item[(ii.2)] Finally, if the condition $(ii.1)$ does not hold true, we have thanks to \eqref{CofacteurBis1},
\begin{eqnarray}\label{finalCase}
x^T\Sigma^{-1}x&=&\sum_{i,j>i_0}\Sigma^{-1}_{ij}x_ix_j+\sum_{i,j\leq i_0}\Sigma^{-1}_{ij}x_ix_j+2\sum_{i< i_0,j\geq i_0}\Sigma^{-1}_{ij}x_ix_j\notag\\
&\geq&\sum_{i,j\geq i_0}\Sigma^{-1}\left((s_{2k-1},s_{2k-2})_{k\in \llbracket i_0+1,m+1\rrbracket}\right)_{ij}x_ix_j-2\left|\sum_{i< i_0,j\geq i_0}\Sigma^{-1}_{ij}x_ix_j\right|\notag\\
&&-|\Theta(\kappa)|\sum_{i,j\geq i_0}\left|\Sigma^{-1}\left((s_{2k-1},s_{2k-2})_{k\in \llbracket i_0+1,m+1\rrbracket}\right)_{ij}x_ix_j\right|\notag\\
\end{eqnarray}

Using the induction hypothesis, we have that 
\begin{eqnarray*}\sum_{i,j\geq i_0}\Sigma^{-1}\left((s_{2k-1},s_{2k-2})_{k\in \llbracket i_0+1,m+1\rrbracket}\right)_{ij}x_ix_j&\geq &D(m)\sum_{i=i_0+1}^{m+1}\frac{x_i^2}{|s_{2i-1}-s_{2i-2}|^{2H}}\\
&\geq&\frac{D(m)}{2}\sum_{i=1}^{m+1}\frac{x_i^2}{|s_{2i-1}-s_{2i-2}|^{2H}}.
\end{eqnarray*}

We also have, thanks to \eqref{CofacteurBis2},
\begin{eqnarray*}
&&\Theta(\kappa)\sum_{i,j\geq i_0}\left|\Sigma^{-1}\left((s_{2k-1},s_{2k-2})_{k\in \llbracket i_0+1,m+1\rrbracket}\right)_{ij}x_ix_j\right|\\&\leq& \Theta(\kappa)K(m-i_0)\sum_{i=i_0+1}^{m+1}\frac{x_i^2}{|s_{2i-1}-s_{2i-2}|^{2H}}\\
&\leq&\frac{D(m)}{16(m+1)}\sum_{i=i_0+1}^{m+1}\frac{x_i^2}{|s_{2i-1}-s_{2i-2}|^{2H}},
\end{eqnarray*}
where the last line follows from the definition of $\kappa$.

 Similarly, we have 
\begin{eqnarray*}
\left|2\sum_{i< i_0,j\geq i_0}\Sigma^{-1}_{ij}x_ix_j\right|&\leq& 2\sum_{i\leq i_0,j>i_0}\frac{K|x_ix_j|}{|s_{2i-1}-s_{2i-2}|^H|s_{2j-1}-s_{2j-2}|^{H}}\\
&\leq&2(m+1)^2\max_{i> i_0}\frac{x_i^2}{|s_{2i-1}-s_{2i-2}|^{2H}}\sqrt{\kappa}\\
&\leq&2(m+1)^2\sqrt{\kappa}\sum_{i=1}^{m+1}\frac{x_i^2}{|s_{2i-1}-s_{2i-2}|^{2H}}\\
&\leq& \frac{D(m)}{8}\sum_{i=1}^{m+1}\frac{x_i^2}{|s_{2i-1}-s_{2i-2}|^{2H}},
\end{eqnarray*}
where the third inequality follows from the fact that the condition \eqref{conditionMax} is not satisfied, and where the last line follows from the definition of $\kappa$.
 Plugging all these elements into \eqref{finalCase} yields
$$x^T\Sigma^{-1}x\geq \frac{D(m)}{4}\sum_{i=1}^{m+1}\frac{x_i^2}{|s_{2i-1}-s_{2i-2}|^{2H}}.$$
\end{enumerate}
In every case, the result is valid for $k=m+1$, which completes the proof by induction.
\end{proof}

In order to prove Proposition \ref{mainBound}, we will also need the following result whose proof is trivial, and hence omitted.
\begin{lemma}\label{trivial}
Let $x\in\mathbb{R}$. Then
$ |e^x-1|\leq |x|e^{x\vee 0}.$
\end{lemma}

We are now ready to complete the proof of Proposition \ref{mainBound}.
As in the statements of Lemma \ref{decompDet} and Proposition \ref{mainBound}, let us write\\
 $\Sigma:=\Sigma(\left(t_i,t_{i-1}\right))_{i\in\llbracket 1,p+q\rrbracket}$, and
 $\Sigma':=\Sigma\left(\left(t_{\bar J[k-1]},t_{\bar J[k]}\right)_{k\in\llbracket 1,q\rrbracket}\right)$ to simplify the expressions.
 Recall also that $\Theta$ is a generic function which satisfies $|\Theta(h)| \leq \beta h^{2-2H}$.
 Moreover let us write $A$ (resp $A'$) for the $q+p$ (resp $q$)-dimensional vector whose first coordinate is $a$ and the other ones are $0$.

\begin{proof}[Proof\nopunct]\textit{of \eqref{Basic}:}
\begin{eqnarray*}&&\mathbb{E}\left[F\left(B_{t_1}-a,\ldots, B_{t_{p+q}}-a\right)\right]\\
&=&\frac{1}{(2\pi)^\frac{p+q}{2}\sqrt{det(\Sigma)}}\int_{\mathbb{R}^{p+q}}e^{-\frac{y^T\Sigma^{-1}y}{2}}dy_1\ldots dy_{p+q}\\
&&\times F\left(y_1-a,\ldots,y_1-a+\sum_{k=2}^iy_k,\ldots,y_1-a+\sum_{k=2}^{p+q}y_k\right)\\
&=&\frac{1}{(2\pi)^\frac{p+q}{2}\sqrt{det(\Sigma)}}\int_{\mathbb{R}^{p+q}}e^{-\frac{(y+A)^T\Sigma^{-1}(y+A)}{2}}\\
&&\times F\left(y_1,\ldots,y_1+\sum_{k=2}^iy_k,\ldots,y_1+\sum_{k=2}^{p+q}y_k\right)dy_1\ldots dy_{p+q},
\end{eqnarray*}
where we have performed the change of variable $y_1\rightarrow y_1-a$.

 We can now apply Lemma \ref{Diagonalisation} and the the local non-determinism property of the fractional Brownian motion (see Lemma \ref{Sandwich}) to obtain
\begin{eqnarray*}
&&\left|\mathbb{E}\left[F\left(B_{t_1}-a,\ldots, B_{t_{2p+q}}-a\right)\right]\right|\\
&\leq&\frac{\beta}{\left(2\pi\right)^\frac{p+q}{2}\prod_{i=1}^{p+q}|t_i-t_{i-1}|^H}\int_{\mathbb{R}^{p+q}}e^{-D\frac{(y_1+a)^2}{|t_1|^{2H}}-D\sum_{k=2}^{p+q}\frac{y_k^2}{|t_k-t_{k-1}|^{2H}}}\\
&&\times \left|F\left(y_1,\ldots,y_1+\sum_{k=2}^iy_k,\ldots,y_1+\sum_{k=2}^{p+q}y_k\right)\right|dy_1\ldots dy_{p+q}.
\end{eqnarray*}
We can now exploit hypothesis $(H_1)$: indeed, when $|y_1|\leq M\sum_{i=1}^p|y_{J[i]}|$, we have that
\begin{eqnarray*}
(a+y_1)^2\leq \frac{a^2}{4}\implies \sum_{i=1}^p\frac{y^2_{J[i]}}{|t_{J[i]}-t_{J[i]-1}|^{2H}}\geq\sum_{i=1}^p\frac{y^2_{J[i]}}{2|t_{J[i]}-t_{J[i]-1}|^{2H}}+\frac{y_1^2}{2pMT^{2H}}\\
\geq \sum_{i=1}^p\frac{y^2_{J[i]}}{2|t_{J[i]}-t_{J[i]-1}|^{2H}}+\frac{a^2}{8pMT^{2H}}.
\end{eqnarray*}
This leads to a bound
\begin{eqnarray*}
&&\left|\mathbb{E}\left[F\left(B_{t_1}-a,\ldots, B_{t_{2p+q}}-a\right)\right]\right|\\
&\leq&\frac{1}{\left(2\pi\right)^\frac{p+2}{2}\prod_{i=1}^{p+q}|t_i-t_{i-1}|^H}\int_{\mathbb{R}^{p+q}}e^{-D'a^2-D'\sum_{i=1}^{p}\frac{y^2_{J[i]}}{2|t_{J[i]}-t_{J[i]-1}|^{2H}}}\\
&&\times \left|F\left(y_1,\ldots,y_1+\sum_{k=2}^iy_k,\ldots,y_1+\sum_{k=2}^{p+q}y_k\right)\right|dy_1\ldots dy_{p+q},
\end{eqnarray*}
with $D'=\frac{D}{8pMT^{2H}}\wedge \frac{1}{4T^{2H}}$. This concludes the proof of \eqref{Basic}.
\end{proof}

\begin{proof}[Proof\nopunct]\textit{of \eqref{complete}:}
Recall that $1\in \bar J$. We can now write, using again the change of variables $y_1\rightarrow y_1-a$,

\begin{eqnarray*}
&&\mathbb{E}[F(B_{t_1}-a,\ldots, B_{t_{p+q}}-a)]\\
&=&\frac{1}{(2\pi)^\frac{p+q}{2}\sqrt{det(\Sigma)}}\int_{\mathbb{R}^{2p+q}}\exp\left({-\frac{\sum_{i=1}^p y^2_{J[i]}\Sigma^{-1}_{J[i],J[i]}}{2}}\right)\\
&&\times\exp\left({-\frac{\sum_{i,j=1}^{p+q} y_iy_j\Sigma^{-1}_{i,j}}{2}}\mathbb{I}_{(i,j)\in\mathcal D}\right)\\
&&\times F\left(y_1-a,\ldots,y_1-a+\sum_{k=2}^iy_k,\ldots,y_1-a+\sum_{k=2}^{p+q}y_k\right)dy_1\ldots dy_{p+q}\\
&=&\frac{1}{(2\pi)^\frac{p+q}{2}\sqrt{det(\Sigma)}}\int_{\mathbb{R}^{p+q}}\exp\left({-\frac{\sum_{i=1}^p y^2_{J[i]}\Sigma^{-1}_{J[i],J[i]}}{2}}\right)\\
&&\times\exp\left({-\frac{\sum_{i,j=1}^{p+q} (y+A)_i(y+A)_j\Sigma^{-1}_{i,j}}{2}}\mathbb{I}_{(i,j)\in\mathcal D}\right)\\
&&\times F\left(y_1,\ldots,y_1+\sum_{k=2}^iy_k,\ldots,y_1+\sum_{k=2}^{p+q}y_k\right)dy_1\ldots dy_{p+q}\\
\end{eqnarray*}
where the notation $\mathcal D$ is introduced in \eqref{D}.

 From the first item in Lemma \ref{decompDet}, we have
$$\frac{1}{\sqrt{det(\Sigma)}}=\frac{(1+\Theta(h))}{\sqrt{\det\left(\Sigma'\right)}\prod_{i=1}^p|t_{J^1_{\sigma}[i]}-t_{J^1_{\sigma}[i]-1}|^H}.$$
We then have 
\begin{eqnarray*}&&\mathbb{E}\left[F\left(B_{t_1}-a,\ldots, B_{t_{p+q}}-a\right)\right]\\
&=&\frac{d_{\Sigma'}(a,0,\ldots,0)}{(2\pi)^\frac{q}{2}\sqrt{det(\Sigma')}}\\
&&\times\int_{\mathbb{R}^{p+q}}\frac{\exp\left({-\sum_{i=1}^p\frac{ y^2_{J[i]}}{2|t_{J[i]}-t_{J[i]-1}|^{2H}}}\right)}{(2\pi)^\frac{p}{2}\prod_{i=1}^p|t_{J[i]}-t_{J[i]-1}|^H}\\
&&\times F\left(y_1,\ldots,y_1+\sum_{k=2}^iy_k,\ldots,y_1+\sum_{k=2}^{p+q}y_k\right)\\
&&\times \left(1+\Theta(h)\right)Rdy_1\ldots dy_{p+q}.
\end{eqnarray*}
Here
\begin{eqnarray*}
R&=&\exp\left(\frac{-(y+A)^T\Sigma^{-1}(y+A)+(u(y)+A')\Sigma'^{-1}(u(y)+A')}{2}\right.\\
&&\left.+\sum_{i=1}^p\frac{y^2_{J[i]}}{2|t_{J[i]}-t_{J[i]-1}|^{2H}}\right),
\end{eqnarray*}
where the notations $d_{M}$ has been introduced in \eqref{densityExp}, and $u(y)$ is the $q$-dimensional vector verifying 
$$\forall i\in\llbracket 1,q\rrbracket, ~u(y)_i:=y_{\bar J[i]}.$$
Hence we obtain
\begin{eqnarray*}
&&\left|\mathbb{E}[F(B_{t_1}-a,\ldots, B_{t_{p+q}}-a)]-\frac{d_{\Sigma'}(A')}{(2\pi)^\frac{q}{2}\sqrt{det(\Sigma')}}\right.\\
&&\left.\times\int_{\mathbb{R}^{p+q}}\mathbb{E}[F(\tilde B(y))]dy_1\ldots dy_{p+q}\right|\\
&\leq& e_1+e_2,
\end{eqnarray*}
with
\begin{eqnarray*}
e_1&:=&\frac{d_{\Sigma'}(A')}{(2\pi)^\frac{q}{2}\sqrt{det(\Sigma')}}\int_{\mathbb{R}^{p+q}}\left|1-R\right|\left|F\left(y_1,\ldots, y_{1}+\sum_{k=2}^{p+q}y_k\right)\right|\\
&&\times \frac{\exp\left({-\sum_{i=1}^p\frac{ y^2_{J[i]}}{2|t_{J[i]}-t_{J[i]-1}|^{2H}}}\right)}{(2\pi)^\frac{p}{2}\prod_{i=1}^p|t_{J[i]}-t_{J[i]-1}|^H}dy_{1}\ldots dy_{p+q},
\end{eqnarray*}
and
\begin{eqnarray*}
e_2&:=&\frac{d_{\Sigma'}(A')}{(2\pi)^\frac{q}{2}\sqrt{det(\Sigma')}}\Theta(h)\int_{\mathbb{R}^{p+q}}R\left|F\left(y_1,\ldots, y_{1}+\sum_{k=2}^{p+q}y_k\right)\right|\\
&&\times \frac{\exp\left({-\sum_{i=1}^p\frac{ y^2_{J[i]}}{2|t_{J[i]}-t_{J[i]-1}|^{2H}}}\right)}{(2\pi)^\frac{p}{2}\prod_{i=1}^p|t_{J[i]}-t_{J[i]-1}|^H}dy_{1}\ldots dy_{p+q}.
\end{eqnarray*}
We will only study the term $e_1$, while the bound for $e_2$ is easier and can be obtained similarly, leading to
\begin{eqnarray}\label{D21}e_2&\leq& \beta\theta(h)e^{-Da^2}\int_{\mathbb{R}^{p+q}}\left|F\left(y_1,\ldots, y_{1}+\sum_{k=2}^{p+q}y_k\right)\right|\notag\\
&&\times \exp\left(-D\sum_{i=1}^{p}\frac{y_{J[i]}^2}{|t_{J[i]}-t_{J[i]-1}|^{2H}}\right)dy_1\ldots dy_{p+q},\notag\\
\end{eqnarray}
for $\beta,D>0$ independent of $t_1,\ldots,t_{p+q}$. For the term $e_1$, notice that $R$ is written as $R=e^{R'}$ so thanks to Lemma \ref{trivial}, $|1-R|\leq |R'|e^{R'\vee 0}$. Thus
\begin{eqnarray*}
&&|1-R|\\&\leq&\left|\frac{1}{2}(y+A)^TM^{-1}(y+A)-\sum_{i=1}^{p}\frac{y_{J^1_{\sigma}[i]}}{2|t_{J^1_{\sigma}[i]}-t_{J^1_{\sigma}[i]-1}|^{2H}}\right.\\
&&\left.-\frac{1}{2}\left(u(y)+A'\right)M'^{-1}\left(u(y)+A'\right)\right|e^{R'\vee 0}
\end{eqnarray*}
We now distinguish two cases depending on whether $R'>0$ or $R'\leq 0$:

\begin{itemize}
\item if $R'<0$, then 
\begin{eqnarray*}
&&d_{\Sigma'}(A')\exp\left({-\sum_{i=1}^p\frac{ y^2_{J[i]}}{2|t_{J[i]}-t_{J[i]-1}|^{2H}}}\right)e^{R'\vee 0}\\
&=&\exp\left(-\frac{1}{2}\left(u(y)+A'\right)\Sigma'^{-1}\left(u(y)+A'\right)-\sum_{i=1}^{p}\frac{y^2_{J[i]}}{2|t_{J[i]}-t_{J[i]-1}|^{2H}}\right)\\
&\leq&\exp\left(-D\left(\sum_{i=1}^q\frac{\left(A'_{\bar J[i]}+y_{\bar J[i]}\right)^2}{2|t_{\bar J[i]}-t_{\bar J[i]-1}|^{2H}}+\sum_{i=1}^{p}\frac{y^2_{J[i]}}{2|t_{J[i]}-t_{J[i]-1}|^{2H}}\right)\right).
\end{eqnarray*}
where the last line comes from Lemma \ref{Diagonalisation}.

\item If $R'>0$, we similarly have
\begin{eqnarray*}
&&d_{\Sigma'}(A')\exp\left({-\sum_{i=1}^p\frac{ y_{J[i]}^2}{2|t_{J[i]}-t_{J[i]-1}|^{2H}}}\right)e^{R'\vee 0}\\
&\leq& \exp\left(-D\sum_{i=1}^{p+q}\frac{(y_i+A_i)^2}{|t_i-t_{i-1}|^{2H}}\right)\\
&\leq&\exp\left(-D\left(\sum_{i=1}^q\frac{\left(A'_{\bar J[i]}+y_{\bar J[i]}\right)^2}{|t_{\bar J[i]}-t_{\bar J[i]-1}|^{2H}}+\sum_{i=1}^{p}\frac{y_{J[i]}^2}{|t_{J[i]}-t_{J[i]-1}|^{2H}}\right)\right)
\end{eqnarray*}
giving us the same upper bound.
\end{itemize}
Next, exploiting $(H_1)$ and proceeding as in the proof of \eqref{Basic} above, we obtain that for some $D>0$, independent on the choice of $t_i$, 
\begin{eqnarray*}
&&\left|F\left(y_1,\ldots, y_{1}+\sum_{k=2}^{p+q}y_k\right)\right|d_{\Sigma'}(A')\exp\left({-\sum_{i=1}^p\frac{ y_{J[i]}^2}{2|t_{J[i]}-t_{J[i]-1}|^{2H}}}\right)e^{R'\vee 0}\\
&\leq&\left|F\left(y_1,\ldots, y_{1}+\sum_{k=2}^{p+q}y_k\right)\right|\exp\left(-Da^2-D\sum_{i=1}^p\frac{y^2_{J[i]}}{|t_{J[i]}-t_{J[i]-1}|^{2H}}\right).
\end{eqnarray*}
It remains to bound the term $R'$. Exploiting the items (b.), (c.) and (d.) in Lemma \ref{decompDet}, we have 
\begin{eqnarray*}
&&\left|\frac{1}{2}(y+A)^T\Sigma^{-1}(y+A)-\sum_{i=1}^{p}\frac{y^2_{J^1_{\sigma}[i]}}{2|t_{J^1_{\sigma}[i]}-t_{J^1_{\sigma}[i]-1}|^{2H}}\right.\\
&&\left.-\frac{1}{2}\left(u(y)+A'\right)\Sigma'^{-1}\left(u(y)+A'\right)\right|\\
&\leq&\frac{1}{2}\sum_{i=1}^{p+q}|(y_i+A_i)(y_j+A_j)||\Sigma^{-1}_{ij}|\mathbb{I}_{(i,j)\in\mathcal D}\\
&&+\frac{1}{2}\sum_{i\neq j=1}^p|y_{J[i]}y_{J[j]}||\Sigma^{-1}_{J[i],J[j]}|+\Theta(h)\frac{1}{2}\sum_{i=1}^{p}|y_{J[i]}|^2\frac{1}{|t_{J[i]}-t_{J[i]-1}|^{2H}},
\end{eqnarray*}
with
\begin{eqnarray*}&&|\Sigma^{-1}_{ij}|\mathbb{I}_{(i,j)\in\mathcal D}\leq\frac{\beta}{\max\left(|t_i-t_{i-1}||t_j-t_{j-1}|^{2H-1}+|t_j-t_{j-1}||t_i-t_{i-1}|^{2H-1}\right)},\\
&&|\Sigma^{-1}_{J[i],J[j]}|\leq\Theta(h)\frac{1}{|t_{J[i]}-t_{J[j]-1}|^H|t_{J[i]}-t_{J[i]-1}|^H}.
\end{eqnarray*}
For the term
$$
\frac{1}{2}\sum_{i\neq j=1}^p|y_{J[i]}y_{J[j]}||\Sigma^{-1}_{J[i],J[j]}|+\Theta(h)\frac{1}{2}\sum_{i=1}^{p}|y_{J[i]}|^2\frac{1}{|t_{J[i]}-t_{J[i]-1}|^{2H}},
$$
we use identities $|ab|\leq\frac{a^2+b^2}{2}$ and $a^2e^{-a^2}\leq e^{\frac{-a^2}{2}}$ and obtain that there exists two constants $\beta,D>0$, independent of the choice of $t_i$, such that 
\begin{eqnarray}\label{D11}
&&\left|F\left(y_1,\ldots, y_{1}+\sum_{k=2}^{p+q}y_k\right)\right|d_{\Sigma'}(A')\exp\left({-\sum_{i=1}^p\frac{ y_{J[i]}^2}{2|t_{J[i]}-t_{J[i]-1}|^{2H}}}\right)e^{R'\vee 0}\notag\\
&&\times \left(\frac{1}{2}\sum_{i\neq j=1}^p|y_{J[i]}y_{J[j]}||\Sigma^{-1}_{J[i],J[j]}|+\Theta(h)\frac{1}{2}\sum_{i=1}^{p}|y_{J[i]}|^2\frac{1}{|t_{J[i]}-t_{J[i]-1}|^{2H}}\right)\notag\\
&\leq&\Theta(h)e^{-Da^2-D\sum_{i=1}^p\frac{y^2_{J[i]}}{|t_{J[i]}-t_{J[i]-1}|^{2H}}}.\notag\\
\end{eqnarray}
For the term 
$$
\frac{1}{2}\sum_{i=1}^{p+q}|(y_i+A_i)(y_j+A_j)||\Sigma^{-1}_{ij}|\mathbb{I}_{(i,j)\in\mathcal D},
$$
there are three cases to distinguish: $(i,j)\in\bar J\times\bar J$,  $(i,j)\in\bar J\times J$, and $(i,j)\in J\times\bar J$.
 If $(i,j)\in J\times\bar J$, then for all $l\in\llbracket 1,p\rrbracket$ it holds that
$$|\Sigma^{-1}_{ij}|\leq \frac{\beta}{|t_j-t_{j-1}||t_i-t_{i-1}|^{2H-1}}\leq\frac{\beta}{h}\frac{1}{|t_i-t_{i-1}|^H|t_{J[l]}-t_{J[l]-1}|^H}.$$
Then, using also $(H_1)$, we get 
\begin{eqnarray}\label{N1}
&&\left|F\left(y_1,\ldots, y_{1}+\sum_{k=2}^{p+q}y_k\right)\right||\Sigma^{-1}_{i,j}||y_{i}y_j|\notag\\
&\leq&h\left|F\left(y_1,\ldots, y_{1}+\sum_{k=2}^{p+q}y_k\right)\right|\sum_{l=1}^pM\beta\frac{|y_iy_{J[l]}|}{|t_i-t_{i-1}|^H|t_{J[l]}-t_{J[l]-1}|^H}.\notag\\
\end{eqnarray}
The case $(i,j)\in \bar J\times J$ can be treated with a symmetric argument, by intechanging the roles of $i$ and $j$. Finally, when $(i,j)\in\bar J\times\bar J$ we obtain
\begin{eqnarray}\label{N2}
&&\left|F\left(y_1,\ldots, y_{1}+\sum_{k=2}^{p+q}y_k\right)\right||\Sigma^{-1}_{i,j}||y_{i}y_j|\notag\\
&\leq&h^{2H}\left|F\left(y_1,\ldots, y_{1}+\sum_{k=2}^{p+q}y_k\right)\right|\sum_{l,r=1}^pM^2\beta\frac{|y_iy_{J[l]}|}{|t_{J[r]}-t_{J[r]-1}|^H|t_{J[l]}-t_{J[l]-1}|^H}.\notag\\
\end{eqnarray}
To conclude, we have obtained 
\begin{eqnarray}
&&\left|F\left(y_1,\ldots, y_{1}+\sum_{k=2}^{p+q}y_k\right)\right|d_{\Sigma'}(A')\exp\left({-\sum_{i=1}^p\frac{ y_{J[i]}^2}{2|t_{J[i]}-t_{J[i]-1}|^{2H}}}\right)e^{R'\vee 0}\notag\\
&&\times \frac{1}{2}\sum_{i, j=1}^{p+q}|y_{i}y_{j}||\Sigma^{-1}_{ij}|\mathbb{I}_{(i,j)\in\mathcal D}\notag\\
&\leq&\beta(h+h^{2H})e^{-Da^2-D\sum_{i=1}^p\frac{y^2_{J[i]}}{|t_{J[i]}-t_{J[i]-1}|^{2H}}}.\label{D12}
\end{eqnarray}
for some $D,\beta>0$. Now combining \eqref{D21}, \eqref{D11}, \eqref{D12} together with the fact that $h,h^{2H}\leq h^{2-2H}$ if $H>\frac{1}{2}$ and $h\in (0,1)$ yields the desired result \eqref{complete}. This completes the proof.
\end{proof}

\appendix

\section{Some technical results}
In this section, we prove three lemmas which are useful for the proof of Theorem \ref{main}.

\begin{lemma}\label{calculus1}
Let $\theta^2>0$ and $X\sim\mathcal N(0,\theta^2)$. Then, for all $\alpha>0$,
$$\mathbb{E}\left[\int_{\mathbb{R}}|y+X-\alpha|\mathbb{I}_{sgn(y+X-\alpha)\neq sgn(y-\alpha)}dy\right]=\frac{1}{2}\theta^2.$$
\end{lemma}

\begin{proof}
We do the change of variable $y\rightarrow y-\alpha$ to obtain
\begin{eqnarray*}
&&\mathbb{E}\left[\int_{\mathbb{R}}|y+X-\alpha|\mathbb{I}_{sgn(y+X-\alpha)\neq sgn(y-\alpha)}dy\right]\\
&=&\mathbb{E}\left[\int_{\mathbb{R}}|y+X|\mathbb{I}_{(y+X)>0,y<0}dy\right]+\mathbb{E}\left[\int_{\mathbb{R}}|y+X|\mathbb{I}_{(y+X)<0,y>0}dy\right]\\
&=&2\mathbb{E}\left[\int_{\mathbb{R}}|y+X|\mathbb{I}_{(y+X)>0,y<0}dy\right].
\end{eqnarray*}
Here, by using also Fubini's theorem,
\begin{eqnarray*}
&&\mathbb{E}\left[\int_{\mathbb{R}}|y+X|\mathbb{I}_{(y+X)>0,y<0}dy\right]\\
&=&\frac{1}{\sqrt{2\pi}\theta}\int_{-\infty}^0\int_{-y}^{\infty}(x+y)e^{-\frac{x^2}{2\theta^2}}dxdy\\
&=&\frac{1}{\sqrt{2\pi}\theta}\left(\int_{0}^{\infty}\int_y^{\infty}xe^{-\frac{x^2}{2\theta^2}}dxdy-\int_0^{\infty}e^{-\frac{x^2}{2\theta^2}}\int_0^xydydx\right)\\
&=&\frac{1}{\sqrt{2\pi}\theta}\left(\int_0^{\infty}\theta^2e^{-\frac{y^2}{2\theta^2}}dy-\frac{1}{2}\int_0^{\infty}x^2e^{-\frac{x^2}{2\theta^2}}dx\right)\\
&=&\theta^2\mathbb{P}[X>0]-\frac{1}{2}\mathbb{E}[X^2\mathbb{I}_{X>0}]=\frac{1}{4}\theta^2,
\end{eqnarray*}
which concludes the proof.
\end{proof}

\begin{lemma}\label{calculus2}
Let $X_1,X_2$ be independent centered Gaussian variables with variances $\theta_1^2$ and $\theta_2^2$. Then, for all $\alpha>0$, we have 
$$\mathbb{E}\left[\int_{\mathbb{R}}\left|\mathbb{I}_{y>\alpha}-\mathbb{I}_{y+X_2>\alpha}\right|dy\right]\leq |\theta_2|$$
and 
$$\mathbb{E}\left[\int_{\mathbb{R}^2}\left|\left(\mathbb{I}_{y_1>\alpha}-\mathbb{I}_{y_1+X_1>\alpha}\right)\left(\mathbb{I}_{y_1+X_1+y_2>\alpha}-\mathbb{I}_{y_1+X_1+y_2+X_2>\alpha}\right)\right|dy_2dy_1\right]\leq |\theta_1\theta_2|.$$
\end{lemma}

\begin{proof}
Performing change of variable $y_1\rightarrow y_1-\alpha$ and conditioning with respect to $X_1$ we get
\begin{eqnarray*}
&&\mathbb{E}\left[\int_{\mathbb{R}^2}\left|\left(\mathbb{I}_{y_1>\alpha}-\mathbb{I}_{y_1+X_1>\alpha}\right)\left(\mathbb{I}_{y_1+X_1+y_2>\alpha}-\mathbb{I}_{y_1+X_1+y_2+X_2>\alpha}\right)\right|dy_2dy_1\right]\\
&\leq&\mathbb{E}\left[\int_{\mathbb{R}^2}\left|\mathbb{I}_{y_1>0}-\mathbb{I}_{y_1+X_1>0}\right|\mathbb{E}\left[\left|\mathbb{I}_{y_2>-(y_1+X_1)}-\mathbb{I}_{y_2+X_2>-(y_1+X_1)}\right||X_1\right]dy_2dy_1\right]\\
\end{eqnarray*}

Proceeding as in Lemma \ref{calculus1}, we can see that for all $\beta>0$, 
$$\int_{\mathbb{R}}\mathbb{E}\left[|\mathbb{I}_{y>\beta}-\mathbb{I}_{y+X_2>\beta}|\right]dy\leq|\theta_2|,$$
which prove the first part of the statement.

 Then,
\begin{eqnarray*}
&&\left|\mathbb{E}\left[\int_{\mathbb{R}^2}\left(\mathbb{I}_{y_1>\alpha}-\mathbb{I}_{y_1+X_1>\alpha}\right)\left(\mathbb{I}_{y_1+X_1+y_2>\alpha}-\mathbb{I}_{y_1+X_1+y_2+X_2>\alpha}\right)dy_2dy_1\right]\right|\\
&\leq&\mathbb{E}\left[\int_{\mathbb{R}}\left|\mathbb{I}_{y_1>0}-\mathbb{I}_{y_1+X_1>0}\right||\theta_2|dy_1\right]\\
&\leq&|\theta_1\theta_2|,
\end{eqnarray*}
which concludes the proof.
\end{proof}

\begin{lemma}\label{calculus3}
Let $\phi$ be as in \eqref{density} and $a\in\mathbb{R}$. There is a constant $\beta>0$ such that we have for all $n\in\mathbb{N}^*$,
\begin{equation}\label{ApproxU}\int_{[0,T]^2}\left|\phi_{u,v}(a,a)-\phi_{u_n,v}(a,a)\right|\mathbb{I}_{(u,v)\in\mathcal C}dudv\leq \beta e^{-\frac{a^2}{4T^{2H}}}n^{1-H},\end{equation}
\begin{equation}\label{ApproxUV}\int_{[0,T]^2}\left|\phi_{u,v}(a,a)-\phi_{u_n,v_n}(a,a)\right|\mathbb{I}_{(u,v)\in\mathcal C}dudv\leq \beta e^{-\frac{a^2}{4T^{2H}}}n^{1-H}
\end{equation}
where $\mathcal{C}=\{(u,v)\in [0,T]^2, \min(u,v,|u-v|)>\frac{2}{n}\}$.
\end{lemma}

\begin{proof} We will only do the proof of \eqref{ApproxU} since the proof of \eqref{ApproxUV} is similar.
 Let us write $d_{u,v}=det(\Sigma(\min(u,v),(\max(u,v)-\min(u,v))))$. By local non-determinism, we have that
 $|d_{u,v}|\geq k_H\min(u,v)^{2H}|u-v|^{2H}$.
 Notice also that since $(u,v)\in\mathcal{C}$, $\frac{|u_n\wedge v|^{2H}(|u-v|\wedge |u-n-v|)^{2H}}{|u\wedge v|^{2H}(|u-v|\vee |u-n-v|)}\leq 4^{2H}$.
 We have
\begin{eqnarray*}
&&\int_{\mathcal{C}}\left|\phi_{u,v}(a,a)-\phi_{u_n,v}(a,a)\right|dudv\\
&\leq&\int_{\mathcal C}\left(e^{-\frac{a^2}{4T^{2H}}}\left|\frac{1}{\sqrt{d_{u,v}}}-\frac{1}{\sqrt{d_{u_n,v}}}\right|\right.\\
&&\left.+\frac{4^H}{k_H\min(u,v)^{H}|u-v|)^{H}}\left(e^{-\frac{a^2|u-v|^{2H}}{4d_{u,v}}}-e^{-\frac{a^2|u_n-v|^{2H}}{4d_{u_n,v}}}\right)\right)dudv.
\end{eqnarray*}

Exploiting Lemma \ref{estimates}, we can see that the function $(u,v)\rightarrow d_{u,v}$ verifies, for $x_1\leq x_2$,
\begin{eqnarray*}|d_{x_1,v}-d_{x_2,v}|&\leq& K|x_1-x_2|\left((x_1\wedge v)^{2H-1}(|v-x_1|\vee|v-x_2|)^{2H}\right.\\
&&\left.+(x_1\wedge v)^{2H}(|v-x_1|\vee|v-x_2|)^{2H-1}\right).
\end{eqnarray*}
Consequently, since the square root function is $\frac{1}{2}$-H\"older continuous, one has
\begin{eqnarray*}
&&\left|\frac{1}{\sqrt{d_{u,v}}}-\frac{1}{\sqrt{d_{u_n,v}}}\right|\\
&\leq&\frac{K}{\sqrt{n}}\frac{(u_n\wedge v)^{H-\frac{1}{2}}(|u-v|\vee|u_n-v|^{H}+(u_n\wedge v)^{H}(|u-v|\vee|u_n-v|^{H-\frac{1}{2}})}{k_H(u_n\wedge v)^{2H}(|u-v|\wedge|u_n-v|)^{2H}}\\
&\leq& \frac{4^HK}{\sqrt{n}}\left(\frac{1}{\min(u,v)^H|u-v|^{H+\frac{1}{2}}}+\frac{1}{\min(u,v)^{H+\frac{1}{2}}|u-v|^{H}}\right).
\end{eqnarray*}

Moreover, we have 
\begin{eqnarray*}
&&\left|e^{-\frac{a^2|u-v|^{2H}}{4d_{u,v}}}-e^{-\frac{a^2|u_n-v|^{2H}}{4d_{u_n,v}}}\right|\\
&=&\left|e^{-\frac{|a||u-v|^{H}}{2\sqrt{d_{u,v}}}}-e^{-\frac{|a||u_n-v|^{H}}{2\sqrt{d_{u_n,v}}}}\right|\left|e^{-\frac{a|u-v|^{H}}{\sqrt{d_{u,v}}}}+e^{-\frac{|a||u_n-v|^{H}}{2\sqrt{d_{u_n,v}}}}\right|\\
&\leq&2\left|\frac{|a||u-v|^H}{2\sqrt{d_{u,v}}}-\frac{|a||u_n-v|^{H}}{2\sqrt{d_{u_n,v}}}\right|\frac{K_H(u\wedge v)^H|u-v|^H}{|a||u-v|}\\
&\leq&K\left|\frac{1}{\sqrt{d_{u,v}}}-\frac{1}{\sqrt{d_{u_n,v}}}\right|+K\frac{n^{-H}}{|u-v|^H},
\end{eqnarray*}
where we exploited the fact that $\forall x>0, e^{-x}\leq\frac{1}{x}$ and the fact that\\ $d_{u,v}\leq K_H(u\wedge v)^{2H}|u-v|^{2H}$.

 Finally,
\begin{eqnarray*}
&&\int_{[0,T]^2}\left|\phi_{u,v}(a,a)-\phi_{u_n,v}(a,a)\right|\mathbb{I}_{(u,v)\in\mathcal C}dudv\\
&\leq&e^{-\frac{a^2}{4T^{2H}}}\int_{\mathcal C}\left(\frac{4^HK}{\sqrt{n}}\left(\frac{1}{\min(u,v)^H|u-v|^{H+\frac{1}{2}}}+\frac{1}{\min(u,v)^{H+\frac{1}{2}}|u-v|^{H}}\right)\right.\\
&&\left.+\frac{1}{n^H\min(u,v)^H|u-v|^{2H}}\right)dudv\\
&\leq&Kn^{1-H}e^{-\frac{a^2}{4T^{2H}}}.
\end{eqnarray*}

\end{proof}
\end{document}